\documentclass[a4paper,10pt]{article}

\usepackage[top=4cm, bottom=3cm, left=3.2cm, right=3.2cm]{geometry}

\usepackage{Macros}

\begin{document}

\title{Car Path Tracking in Traffic Flow Networks with Bounded Buffers at Junctions}

\author{Theresa Dambach\footnotemark[1], \; Simone G\"ottlich\footnotemark[1], \; Stephan Knapp\footnotemark[1]}

\footnotetext[1]{University of Mannheim, Department of Mathematics, 68131 Mannheim, Germany (gooettlich@uni-mannheim.de, stknapp@mail.uni-mannheim.de)}

\date{\today}

\maketitle

\begin{abstract}
This article deals with the modeling for an individual car path through a road network, where the dynamics
is driven by a coupled system of ordinary and partial differential equations.
The network is characterized by bounded buffers at junctions that allow for the interpretation of roundabouts or on-ramps
while the traffic dynamics is based on first-order macroscopic equations of Lighthill-Whitham-Richards (LWR) type.   
Trajectories for single drivers are then influenced by the surrounding traffic and can be tracked by appropriate numerical algorithms.  
The computational experiments show how the modeling framework can be used as navigation device. 
\end{abstract}


\maketitle

\section{Introduction} 

Traffic flow modeling is a current and active research area since mathematical models help to get a deeper insight into the traffic mechanisms. 
Essentially, these models can be divided into microscopic and macroscopic depending on their individual characteristics. Microscopic models consider every vehicle as single agent with different properties specified by its position, speed and acceleration, see e.~g.~\cite{Treiber.2013}. In contrast, macroscopic models include fewer specifications and are only characterized by aggregated variables such as traffic density, flow and average speed, see e.~g.~\cite{Garavello.2016,Garavello.2006}. More precisely, in the case of road networks, the macroscopic dynamics is typically given by hyperbolic conservation laws such that the traffic flow at junctions is distributed according to coupling conditions. Within this framework, also on-ramps or roundabouts can be modeled, cf.~\cite{Piccoli.2006}. In \cite{Goatin.2012,Herty.2009}, a novel type of intersection was introduced considering that junctions may have a storage capacity as buffers. 
The approach is closely related to supply chains or production networks, where goods must wait in a buffer to be processed, see \cite{DApice.2010}. 

Recent publications couple microscopic and macroscopic models to analyze the path of a single vehicle that travels along a road, see \cite{Bretti.2008,DelleMonache.2014,Gottlich.2016}. Unlike strongly-coupled models \cite{DelleMonache.2014}, a weakly-coupled approach \cite{Bretti.2008} assumes that the traffic influences the vehicle but not vice versa when the trajectory of the vehicle is given by an ordinary differential equation. 
In the case of congestion, the right-hand side of this ordinary differential equation becomes discontinuous and we cannot expect a classical solution any more. The underlying theory of such kind of problems is considered in \cite{Colombo.2003b,Colombo.2003}. 

In contrast to already existing models, we provide a formulation to track a car in a network with buffers at junctions, which leads to waiting times whenever a buffer is not empty. At junctions with more than one outgoing road, we introduce different methods to decide which road will be taken next. Moreover, a new numerical approach is presented and the tracking algorithm based on \cite{Bretti.2008} is tailored to a network with buffers. 

The outline of the paper is as follows. Section \ref{sec:Model} is devoted to the introduction of the weakly-coupled network model assuming bounded buffers at junctions. The corresponding numerical analysis as well as a naive and complex tracking algorithm are presented in Section \ref{sec:Analysis}. Section \ref{sec:Results} covers the numerical results. In particular, the performance of the developed algorithms at junctions is numerically investigated in several test networks. Additionally, the tracking algorithms are applied to compare different paths and the resulting arrival time. Conclusions and aspects for future research are outlined in Section \ref{chap:Conclusion}.

\section{Weakly-Coupled Network Model} \label{sec:Model}
We start by presenting the traffic flow model of Herty, Lebacque and Moutari \cite{Herty.2009}. As the speed of a single car is assumed to depend on the traffic density, we refer to the approach in \cite{Colombo.2003b,Colombo.2003}, where only a single road is considered. In Section \ref{sec:Car}, a weakly-coupled model to track the car's path in a network with buffers at junctions is developed.

\subsection{Traffic Flow Model with Bounded Buffers} \label{sec:TFM}
We consider a road network specified by a finite, connected and directed graph $\cG = (\cV,\cE)$ in which the set of nodes $\mathcal{V}$ represents the junctions. The set of edges $\mathcal{E}$ denotes the finite collection of roads modeled by intervals $I_e=[a_e,b_e] \subseteq \R$ of length $L_e=b_e-a_e$ for all $e \in \mathcal{E}$. Further, the functions $A: \cE \rightarrow \cV$ and $\Omega: \cE \rightarrow \cV$ provide the source and target node of an edge. For a given junction $v \in \mathcal{V}$, we call $\delta^+(v)$ the set of all incoming roads and $\delta^-(v)$ of all outgoing roads, respectively.

On each road $e \in \mathcal{E}$, we consider the Lighthill-Whitham and Richards (LWR) model \cite{Lighthill.1955,Richards.1956}
\begin{align*}
	\partial_t \rho_e + \partial_x f(\rho_e) = 0 ,
\end{align*}
where $\rho_e(x,t) \in [0,\rho_{max}]$ denotes the car density and $f$ the corresponding flux. We restrict ourselves to $\rho_{max}=1$ and a flux of the form $f(\rho)=\rho v(\rho)$, which satisfies the following assumption according to \cite{Herty.2009}:
\begin{equation*}
f \in C^2([0,1]), \; f(0) = f(1) = 0  \text{ and } f''<0 \text{ with unique maximum at } \rho=\sigma \in (0,1).
\end{equation*}
In order to describe that a junction $v \in \mathcal{V}$ has limited capacity $r_v^{max}>0$, we introduce a bounded buffer $r_v(t) \in [0,r_v^{max}]$ denoting the total number of cars in the buffer at time $t$.  The buffer load can be modeled by an ordinary differential equation considering the in- and outflow at a junction. We fix appropriate initial density functions $\rho_{e,0}$ together with an initial buffer load $r_{v,0} \in [0,r^{max}_v]$ at time $t_0=0$ and formulate the Cauchy problem at an intersection $v \in \mathcal{V}$ as
\begin{subequations}
\begin{align} 
	&\partial_t \rho_e + \partial_x f(\rho_e) = 0, \label{eq:LWR} \\
	&\rho_e(x,0) = \rho_{e,0}(x),\\
	& r'_v(t) = \sum_{e \in \delta^+(v)} f(\rho_e(b_e,t)) -  \sum_{e \in \delta^-(v)} f(\rho_e(a_e,t)), \label{ODE}\\
	& r_v(0) = r_{v,0}, \label{eq:initialBuffer} 
\end{align} 
\end{subequations}
with $x \in (a_e,b_e),~t>0$ and for all $e \in \delta^+(v) \cup \delta^-(v)$. In \cite{Herty.2009}, the existence of a weak entropy solution at junctions with one incoming and one outgoing road for constant initial data is ensured. Their approach has been extended to allow for an arbitrary number of incoming and outgoing roads and initial density functions of bounded total variation as in \cite{Goatin.2012}. 

The remaining step concerns the exact in- and outflow at junctions. More precisely, we have to prescribe coupling conditions to determine the flux and the corresponding density at both ends of each road $q_e^{in}(t)=f(\rho_e(a_e,t))$ and $q_e^{out}(t)=f(\rho_e(b_e,t))$. To simplify the notation, we use the so-called demand and supply function. For a strictly concave flux function $f:[0,1] \rightarrow \R$ with a unique maximum at $\rho=\sigma$, we call $d : [0,1] \rightarrow \R$ and $s : [0,1] \rightarrow \R$ with
	\begin{align*}
		d(\rho) = \begin{cases}
			f(\rho) & \text{if } ~ \rho \leq \sigma, \\
			f(\sigma) & \text{if } ~ \rho > \sigma
		\end{cases} ~~~~\text{and}~~~~
		s(\rho) = \begin{cases}
			f(\sigma) & \text{if } ~ \rho \leq \sigma, \\
			f(\rho) & \text{if } ~ \rho > \sigma
		\end{cases}
	\end{align*}
	demand and supply function.

\begin{remark}
	The aim is to choose the in- and outflow at a junction as high as possible to maximize the throughput. The demand $d(\rho_e(b_e,t))$ gives the maximum flux that can come from a road $e \in \mathcal{E}$ at time $t$ whereas the supply $s(\rho_e(a_e,t))$ denotes the maximum flux that can enter road $e$.
\end{remark}

\begin{figure}[h]
	\centering
	\begin{minipage}[b]{0.45\textwidth}
		\centering
		\begin{tikzpicture}[
		decoration={
			markings,
			mark=at position 1 with {\arrow[scale=1,black]{latex}};
		},scale=0.7
		]
		\draw [postaction={decorate}] (0,0) -- (2,0);
		\draw  (2.3,0) circle (3mm) node {$v$};
		\draw [postaction={decorate}] (2.6,0) -- (4.6,0.7);
		\draw [postaction={decorate}] (2.6,0) -- (4.6,-0.7);
		
		\coordinate[label=above:$q_1^{out}(t)$] ()  at (1.1,0);
		\coordinate[label=above:$q_2^{in}(t)$] ()  at (3,0.3);
		\coordinate[label=above:$q_3^{in}(t)$] ()  at (3,-1);
	
		\draw [postaction={decorate}] (6+0,0.7) -- (6+2,0);
		\draw [postaction={decorate}] (6+0,-0.7) -- (6+2,0);
		\draw  (6+2.3,0) circle (3mm) node {$v$};
		\draw [postaction={decorate}] (6+2.6,0) -- (6+4.6,0);
		
		\coordinate[label=above:$q_1^{out}(t)$] ()  at (6+1.4,0.3);
		\coordinate[label=above:$q_2^{out}(t)$] ()  at (6+1.4,-1);
		\coordinate[label=above:$q_3^{in}(t)$] ()  at (6+3.2,0);
		\end{tikzpicture}
	\end{minipage}
	\caption{One-to-two junction (left) and two-to-one junction (right)}
	\label{im:junction} 
\end{figure}
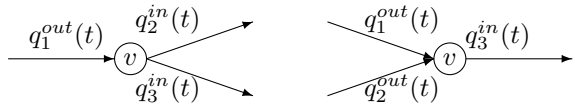 
We only consider networks consisting of dispersing and merging junctions, see Figure \ref{im:junction}. To model the flow at a junction $v$, we also introduce demand and supply functions for the buffer. The demand $d_B$ gives the maximum flux of cars that can leave the buffer and the supply $s_B$ expresses the maximum flux that can enter the buffer.

Further, we assume that the buffer has a constant rate $\mu_v \in \left(0,\max\{ |\delta^+(v)|,|\delta^-(v)| \} f(\sigma)\right)$ at which its load increases and decreases. That means, if the buffer is not completely full, the supply $s_B$ is constant and the maximum flow entering the junction is bounded by $\mu_v$.
Similarly, provided that the buffer is not empty, the demand $d_B$ is constant and the maximum flow leaving the junction cannot exceed $\mu_v$.
\\
\\
\textbf{One-To-Two Junction}
\\
In the case of a dispersing junction, see Figure \ref{im:junction} on the left, distribution rates $\alpha_{2,1} > 0 \text{ and } \alpha_{3,1} > 0$ have to be prescribed. These fixed coefficients are the percentages of drivers coming from the first road and continuing their path on road two or three. Thus, they represent the preferences of drivers. The condition $\alpha_{2,1} + \alpha_{3,1} = 1$ ensures that all drivers are distributed to an outgoing road. 
The demand of the buffer is given by
\begin{align}
	d_B &= \begin{cases}
				\mu_v & \text{if } ~ 0 < r_v(t) \leq r_v^{max}, \\
				\min\{d(\rho_1(b_1,t)),\mu_v\} & \text{if } ~ r_v(t) = 0 \text{.}  \label{eq:a_dB}
		   \end{cases}
\end{align}
The latter is constant and equal to $\mu_v$, if the buffer is not empty and linked to the incoming flux of cars, if $r_v(t)=0$. To include that the maximum number of cars that can leave the junction is distributed to two roads, we restrict $q_2^{in}(t)= f(\rho_2(a_2,t))$ and $q_3^{in}(t)=f(\rho_3(a_3,t))$ by $\alpha_{2,1} d_B$ and $\alpha_{3,1} d_B$ instead of the total possible flux $d_B$, i.e.
\begin{align}
	q_2^{in}(t) &= f(\rho_2(a_2,t)) = \min\{\alpha_{2,1} d_B,s(\rho_2(a_2,t))\}, \label{eq:b_q2}\\
	q_3^{in}(t) &= f(\rho_3(a_3,t)) = \min\{\alpha_{3,1}d_B,s(\rho_3(a_3,t))\} \label{eq:b_q23}\text{.} 
\end{align}
For the buffer supply $s_B$, it must hold that if the buffer is full, i.e. $r_v(t)=r_v^{max}$, the flow entering the junction must not exceed the sum of the two outgoing fluxes or the maximum outflow $\mu_v$. This implies 
\begin{align*}
	s_B &= \min\{ q_2^{in}(t) +  q_3^{in}(t), \mu_v\} \\
	    &= \min\{\alpha_{2,1}\mu_v, s(\rho_2(a_2,t)) \} + \min\{\alpha_{3,1}\mu_v, s(\rho_3(a_3,t)) \} \text{,}
\end{align*}
due to $\alpha_{2,1}\mu_v+\alpha_{3,1}\mu_v = \mu_v$. In the case of an incompletely filled buffer, we assume that the supply is constant and equal to $\mu_v$. These considerations lead to
\begin{align}
	s_B &= \begin{cases}
				\mu_v & \text{if } ~ 0 \leq r_v(t) < r_v^{max}, \\
				\min\{s(\rho_2(a_2,t)),\alpha_{2,1}\mu_v\} + \min\{s(\rho_3(a_3,t)),\alpha_{3,1}\mu_v\} & \text{if } ~ r_v(t) = r_v^{max}.
		   \end{cases} \label{eq:b_sB}
\end{align}
Then, the actual outflow of an incoming road reads as
\begin{align*}
	q_1^{out}(t) &= f(\rho_1(b_1,t)) = \min\{s_B,d(\rho_1(b_1,t))\},
\end{align*}
which means that it is bounded by the supply of the buffer $s_B$ and by the demand of the density on the first road $d(\rho_1(b_1,t))$. The latter guarantees the correct wave speed and the former that the capacity constraints of the buffer are not violated.
~\\
\\
\textbf{Two-To-One Junction}
\\
For a merging junction, see Figure \ref{im:junction} on the right, we need to introduce so-called right-of-way priorities $c_{3,1} > 0,~c_{3,2} > 0$ satisfying $c_{3,1} + c_{3,2} =1$ to guarantee a unique solution at the junction. These parameters are needed to prescribe which cars are allowed to enter whenever the buffer cannot take the total amount caused by restrictions $r_v^{max}$ and $\mu_v$.

The actual outflow towards the buffer depends on the supply of the buffer and on the density at the end of the incoming roads, i.e.
\begin{align}
	q_1^{out}(t) &=f(\rho_1(b_1,t)) = \min\{c_{3,1}s_B,d(\rho_1(b_1,t))\}, \label{eq:c_q1}\\	
	q_2^{out}(t) &=f(\rho_2(b_2,t)) = \min\{c_{3,2}s_B,d(\rho_2(b_2,t))\} \text{,} \label{eq:c_q2}
\end{align}
with buffer supply
\begin{align}
	s_B &= \begin{cases}
				\mu_v & \text{if } ~ 0 \leq r_v(t) < r_v^{max}, \\
				\min\{s(\rho_3(a_3,t)),\mu_v\} & \text{if } ~ r_v(t) = r_v^{max}.
		   \end{cases} 
\end{align}
The actual flux towards road three is given by
\begin{align}
	q_3^{in}(t) &= f(\rho_3(a_3,t)) = \min\{d_B,s(\rho_3(a_3,t))\}. \label{eq:c_q3} 
\end{align}
The demand of the buffer $d_B$, which is the maximum flux that can leave the buffer, is bounded by the two incoming fluxes and $\mu_v$. If $r_v(t)=0$, we obtain $s_B=\mu_v$ and
\begin{align*}
	d_B &= \min \{ q_1^{out}(t) + q_2^{out}(t), \mu_v\} \\
    	&= \min \{c_{3,1}\mu_v,d(\rho_1(b_1,t)) \} + \min\{c_{3,2}\mu_v,d(\rho_2(b_2,t))\}, 
\end{align*}
by using $c_{3,1}\mu_v+c_{3,2}\mu_v = \mu_v$. Together with a constant flow of $\mu_v$, if the buffer is not empty, we finally choose
\begin{align}
	d_B^{Stand} &= \begin{cases}
						\mu_v & \text{if } ~ 0 < r_v(t) \leq r_v^{max}, \\
						\min\{d(\rho_1(b_1,t)),c_{3,1}\mu_v\} + \min\{d(\rho_2(b_2,t)),c_{3,2}\mu_v\} & \text{if } ~ r_v(t) = 0 \text{.} \\
				   \end{cases} \label{eq:c_dB}	
\end{align}
Note that in \cite{Herty.2009} the demand of the buffer is said to be 
\begin{align}
	d_B^{Herty} &= \begin{cases}
						\mu_v & \text{if } ~ 0 < r_v(t) \leq r_v^{max}, \\
						\min\{d(\rho_1(b_1,t)) + d(\rho_2(b_2,t)),\mu_v\} & \text{if } ~ r_v(t) = 0 \text{.} \\
				   \end{cases} \label{eq:c_dB_Paper}
\end{align}
The latter does not ensure a non-negative buffer load $r_v(t)$ if $c_{3,1}\mu_v \gtrless d(\rho_1(b_1,t))$, $c_{3,2}\mu_v \lessgtr d(\rho_2(b_2,t))$ and $s(\rho_3(a_3,t))$ sufficiently large since $d_B$ does not depend on the actual inflow to the buffer but only on the number of cars arriving at the end of the road. To illustrate this problem, we study an example in detail. 
\begin{example} \label{ex:dB}
	Let us assume $f(\rho)=\rho(1-\rho)$, $\mu_v = 0.2$, $c_{3,1} = c_{3,2} = 0.5$ and the initial data $\rho_{1,0} = 0.4,~\rho_{2,0} = 0.1,~\rho_{3,0} = 0.5,~r_{v,0} = 0$. According to \eqref{eq:c_q1}-\eqref{eq:c_q3} and \eqref{eq:c_dB_Paper}, we obtain the fluxes
	\begin{align*}
		q_1^{out}(t) &= \min\{c_{3,1}s_B,d(\rho_{1,0})\} = \min \{ 0.5\cdot 0.2 , 0.24 \}  = 0.1,\\
		q_2^{out}(t) &= \min\{c_{3,2}s_B,d(\rho_{2,0})\} = \min \{ 0.5\cdot 0.2 , 0.09 \} = 0.09,\\
		q_3^{in}(t)  &= \min\{d_B^{Herty},s(\rho_{3,0})\} = \min \{ 0.2,0.25 \} = 0.2,
	\end{align*}
	which lead to a negative buffer load for $r_{v,0} = 0$, i.e.
	\begin{align*}
		r_v'(t) = q_1^{out}(t) + q_2^{out}(t) - q_3^{in}(t) = - 0.01.
	\end{align*}
	In this case, the choice \eqref{eq:c_dB} ensures that the inflow to the buffer equals the outflow with
	\begin{align*}
		q_3^{in}(t) = \min\{d_B^{Stand},s(\rho_{3,0})\} = \min \{ 0.19,0.25 \} = 0.19.
	\end{align*}
	Consequently, $r_v'(t)=0$ and the buffer remains empty as expected. The following lemma provides the connection between the two definitions of $d_B$.
\end{example}

\begin{lemma} \label{lemma:dB}
	Assuming $d(\rho_1(b_1,t)) \neq 0$ and $d(\rho_2(b_2,t)) \neq 0$ as well as time-dependent right-of-way priorities defined by
	\begin{align}
		c_{3,1}(t) = \frac{d(\rho_1(b_1,t))}{d(\rho_1(b_1,t)) + d(\rho_2(b_2,t))}, ~~~~
		c_{3,2}(t) = \frac{d(\rho_2(b_2,t))}{d(\rho_1(b_1,t)) + d(\rho_2(b_2,t))}. \label{eq:c1c2}
	\end{align}
	Then, $d_B^{Herty}$ and $d_B^{Stand}$ provide the same value for the demand of the buffer at a junction with two incoming and one outgoing road.
\end{lemma}

\begin{proof}
	In general, both definitions differ only if $r_v(t) =0$, $c_{3,1}(t)\mu_v \gtrless d(\rho_1(b_1,t))$ and $c_{3,2}(t)\mu_v \lessgtr d(\rho_2(b_2,t))$. Choosing the right-of-way priorities according to \eqref{eq:c1c2} excludes these cases, which means that only $c_{3,1}(t)\mu_v \gtrless d(\rho_1(b_1,t))$ and $c_{3,2}(t)\mu_v \gtrless d(\rho_2(b_2,t))$ is possible. To see this, we assume
	$c_{3,1}(t) \mu_v > d(\rho_1(b_1,t))$ and $c_{3,2}(t) \mu_v < d_2(\rho_2(b_2,t))$ which leads to
	\begin{align*}
		c_{3,1}(t) \mu_v > d(\rho_1(b_1,t))  ~~ & \Leftrightarrow ~~ 0 > d(\rho_1(b_1,t)) \cdot \big( d(\rho_1(b_1,t)) + d(\rho_2(b_2,t)) - \mu_v \big) \\
		                                        & \Leftrightarrow ~~ 0 > d(\rho_1(b_1,t)) + d(\rho_2(b_2,t)) - \mu_v
	\end{align*}
	and
	\begin{align*}
		c_{3,2}(t) \mu_v < d(\rho_2(b_2,t)) ~~ & \Leftrightarrow ~~	0 < d(\rho_2(b_2,t)  \cdot \big( d(\rho_1(b_1,t)) + d(\rho_2(b_2,t)) - \mu_v \big) \\
		                                       & \Leftrightarrow ~~ 0 < d(\rho_1(b_1,t)) + d(\rho_2(b_2,t)) - \mu_v.
	\end{align*}
	Since $d(\rho_1(b_1,t)) >0$ and $ d(\rho_2(b_2,t)) > 0$, we obtain a contradiction. The same reasoning applies to the case $c_{3,1}(t) \mu_v < d(\rho_1(b_1,t))$ and $c_{3,2}(t) \mu_v > d_2(\rho_2(b_2,t))$.
\end{proof}
~\\
\textbf{In- and Outflow}
\\
In a traffic network, we usually introduce some nodes $v \in \mathcal{\tilde{V}},~ \mathcal{\tilde{V}} \subset \mathcal{V}$, with either only one incoming edge or only one outgoing edge. We call edges that are connected to one of these nodes, incoming or outgoing roads, since the cars enter or leave the network at these points. For nodes with only one outgoing edge, we prescribe desired inflow functions $f_v^{in}(t)$. Additionally, we assume unlimited buffers $r_v(t)$ at these nodes, which means $r_v^{max}=\infty$. The load of this buffer will increase when the desired inflow is higher than the capacity of the following road at that time or higher than the maximum possible outflow $\mu_v$. Hence, the inflow towards the first road is restricted to a reasonable value and even if the network is completely blocked, an admissible solution can be obtained. Similar to the buffers inside the network, the load decreases with rate $\mu_v$ when the traffic is less dense and the road is able to absorb a higher flux than $f^{in}_v(t)$. The actual inflow and the demand of the buffer are set as
\begin{align}
	q_e^{in} (t) &= \min\{\tilde{d}_B,s(\rho_e(a_e,t))\}, \label{eq:inflow_q}\\
	\tilde{d}_B &= \begin{cases}
						\mu_v & \text{if } ~ r_v(t) > 0, \\
						\min \{ f_v^{in}(t), \mu_v \} & \text{if } ~ r_v(t)=0.
				   \end{cases} \label{eq:inflow_dB}
\end{align}
Since $f_v^{in}$ is the flux towards the buffer, the latter is given by 
\begin{align*}
	r_v'(t) = f_v^{in}(t) - q_e^{in}(t),~ e \in \delta^-(v) \text{.}
\end{align*}
On outgoing roads, we introduce the absorbing condition
\begin{align}
	q_e^{out} (t) = f(\rho_e(b_e,t)), \label{eq:outflow_f}
\end{align}
which prevents that backwards traveling waves emerge at outgoing nodes. In this case, the flux is not maximized at the end of the edge if $\rho_e(b_e,t)>\sigma$.

\subsection{Car Path} \label{sec:Car}
In the following, we extend the traffic flow model by a microscopic equation to track the path of a single car. 
We develop the mathematical formulation before discussing the choice of the next road in a network based on the car's scope of information about the current traffic situation.

\subsubsection{Model Equations} \label{sec:CarModel}

To track the position $x(t)$ of a single car on every road $e \in \cE$ that the car passes on its way from an initial position to its destination, we need a set of initial value problems
\begin{subequations}
\begin{align}
	\dot{x}(t) &= \omega(\rho_e(x,t)),~ x \in [x_e^*,b_e],~t > t_e^*, \label{eq:CP_net1} \\ 
	x(t^*_e)   &= x^*_e, \label{eq:CP_net2}
\end{align}
\end{subequations}
with an appropriate speed function $\omega(\rho)$ and initial data $(t^*_e,x^*_e)$, which means that the car is located at $x_e^* \in [a_e,b_e]$ on road $e$ at time $t_e^*$. The main challenge is the discontinuous right-hand side of equation \eqref{eq:CP_net1}. In case of congestion, which is described by a shock wave in the LWR model, the speed of the car $\omega(\rho)$ is not continuous. Consequently, the Picard-Lindel\"of theorem cannot be applied to prove existence and uniqueness of a solution. Therefore, in \cite{Colombo.2003b,Colombo.2003} the wider class of Filippov solutions \cite{Filippov.1988} are considered. Such a solution is defined at a discontinuity by the derivative and a solution is not longer required to be continuously differentiable but an absolutely continuous function. For $\omega(\rho)=v(\rho)$, the existence and uniqueness of a Filippov solution have been proven in \cite{Colombo.2003b,Colombo.2003}.

So far,  the car's behavior at junctions is not obvious and we do not know how to choose the initial data $(t^*_e,x^*_e)$ properly. Therefore, we take a closer look at the travel time on a road $e \in \mathcal{E}$ and the waiting time at a junction $v$ with $e \in \delta^+(v)$. The definition of the travel time is based on the formulation in \cite{Gottlich.2016}. Figure \ref{im:Notation} shows an illustration of the notation that will be utilized in the following.

\begin{figure}[h]
	\centering
	\begin{tikzpicture}[
	decoration={
		markings,
		mark=at position 1 with {\arrow[scale=1,black]{latex}};
	},xscale=1.1
	]
	\draw [postaction={decorate}] (0,0) -- (4,0) node [midway,above] {$e$};
	\draw  (4.3,0) circle (3mm) node {$v$};
	\draw [postaction={decorate}] (4.6,0) -- (8.6,0) node [midway,above] {$\tilde{e}$};
	
	\draw[->] (0,-0.7) -- (9,-0.7);
	\coordinate[label=below:$t$] ()  at (8.8,-0.7);
	
	\draw (0,-0.75) -- (0,-0.65);
	\coordinate[label=below:$t_e^*$] ()  at (0,-0.7);
	
	\draw (4,-0.75) -- (4,-0.65);
	\coordinate[label=below:$at_e(t_e^*)$] ()  at (3.5,-0.7);
	
	\draw (4.6,-0.75) -- (4.6,-0.65);
	\coordinate[label=below:$t_{\tilde{e}}^*$] ()  at (4.7,-0.7);

	\draw[decorate,decoration ={brace}, yshift=-4ex] (4,-0.7) -- node[below=0.4ex] {$tt_e(t_e^*)$} (0,-0.7);
	
	\draw[decorate,decoration ={brace}, yshift=-4ex] (4.6,-0.7) -- node[below=0.4ex] {$~~~wt_v(at_e(t_e^*))$} (4,-0.7);
	\end{tikzpicture}
	\caption{Notation of starting, travel, arrival and waiting time}
	\label{im:Notation}
\end{figure}
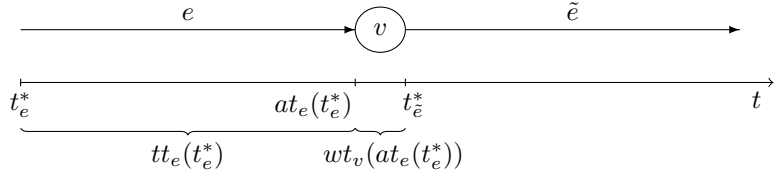

Let us denote by $t_e^* \in [0,T]$ the time at which the car is located at $x_e^* \in [a_e,b_e]$ on a road $e\in \cE$. To simplify the following considerations, we assume without loss of generality that $x_e^*=a_e$. This means that the car is located at the beginning of the considered road, i.e. $x(t_e^*) = x_e^* = a_e$. Hence, the required travel time $tt_e(t_e^*)$ on this road $e$ can be described by the arrival time at the end of the road. Note that the travel time depends on the starting time $t_e^*$ since the traffic density on the road varies over time. Let $\tilde{t}_e$ be defined such that $x(\tilde{t}_e) = b_e$. Then, $\tilde{t}_e$ is strictly greater than $t_e^*$ since all roads have positive length and $x(t)$ is continuous. Additionally, $\tilde{t}_e \leq T$ holds, if the car reaches the end of the road within the time horizon $T$, which means
\begin{align}
	\int_{t_e^*}^{T} \omega(\rho_e(x(t),t)) ~dt \geq b_e-a_e. \label{eq:CP_exceedtt}
\end{align}
Due to the non-negative speed $\omega$, $x(t)$ is non-decreasing. However, it is possible that the car has zero speed in a traffic jam. Therefore, there might be more than one $\tilde{t}_e$ satisfying $x(\tilde{t}_e)=b_e$. For this reason, we define the arrival time at the end of road $e$ by $at_e(t_e^*) = \min \{ t \in ~(t_e^*,T] : x(t)=b_e \}$.
Due to this consideration, the travel time on road $e$ is given by $tt_e(t_e^*) = at_e(t_e^*) - t_e^*$.

Now, we turn to the waiting time at junctions. The cars are assumed to behave according to the First-In-First-Out (FIFO) principle. The idea behind this concept is that if the buffer is non-empty, the considered car has to wait until all cars, which had arrived earlier, have left the buffer. If the buffer load at the junction $v$ is empty when the car arrives at the end of road $e$, which means $r_v(at_e(t_e^*)) = 0$,
the car leaves immediately and the waiting time at node $v$ is equal to zero, $wt_v(at_e(t_e^*)) = 0$.
Since the buffer has no spatial dimension, the end of an incoming road $e$ is the beginning of an outgoing road $\tilde{e}$
\begin{align*}
	b_e = a_{\tilde{e}},~ \text{for all} ~e \in \delta^+(v),~\tilde{e} \in \delta^-(v).
\end{align*}
For this reason, the starting time $t_{\tilde{e}}^*$ on the next road is equal to the arrival time at the end of the previous road $t_{\tilde{e}}^* = at_e(t_e^*)$ and the car must be located at the beginning of the new road $x(t_{\tilde{e}}^*) = b_e = a_{\tilde{e}}$. Note that the road $\tilde{e}$ is only uniquely determined if the junction has only one outgoing road. In case of a dispersing junction, we will need a criteria to decide which road is taken by the driver. This issue is dealt with in the next subsection.

If the buffer is non-empty, i.e. $r_v(at_e(t_e^*)) > 0$, the considered car has to wait and the starting time $t_{\tilde{e}}^*$ on the new road $\tilde{e}$ is defined by the buffer load at the arrival time $at_e(t_e^*)$ and the outflow $f_v^{out}(t)$ while the car is inside the buffer. In fact, we have to solve
\begin{align*}
	r_v(at_e(t_e^*)) = \int_{at_e(t_e^*)}^{t_{\tilde{e}}^*} f^{out}_v(t) ~dt. 
\end{align*}
In the same manner as before, we can argue that $t_{\tilde{e}}^* \in (at_e(t_e^*),T]$ exists if the total outflow up to $T$ exceeds the buffer load at time $at_e(t_e^*)$, i.e.
\begin{align}
	r(at_e(t_e^*)) \leq \int_{at_e(t_e^*)}^{T} f^{out}_v(t) ~dt. \label{eq:CP_exceedwt}
\end{align}
Note that $f^{out}_v(t)$ denotes the total outflow from node $v$ towards all outgoing roads $e \in \delta^-(v)$, i.e.
\begin{align*}
	f^{out}_v(t) = \sum_{e \in \delta^-(v)} q_e^{in}(t).
\end{align*}
This function is non-negative and we are searching a $t_{\tilde{e}}^*$ that satisfies
\begin{align*}
	t_{\tilde{e}}^* = \min \Big\{ t \in (at_e(t_e^*),T] : r(at_e(t_e^*)) = \int_{at_e(t_e^*)}^{t} f^{out}_v(s) ~ds  \Big\}.
\end{align*}
Then, the waiting time is given by $wt_v(at_e(t_e^*)) = t_{\tilde{e}}^* - at_e(t_e^*)$. For all $t \in [at_e(t_e^*),t_{\tilde{e}}^*]$, the car remains at the end of the edge and $x(t) = b_e=a_{\tilde{e}}$.

Now, we are able to define appropriate initial data $(t_e^*,x_e^*)$ for the set of initial value problems \eqref{eq:CP_net1}-\eqref{eq:CP_net2} depending on the travel and waiting times. To achieve this, we prescribe the position of the car's departure $x_{e^*}^* \in [a_{e^*},b_{e^*}]$ on an initial edge $e^* \in \cE$ at initial time $t_{e^*}^* \in [0,T]$. 
If the considered time horizon $T$ is sufficiently large, the initial data for the remaining ODEs can be determined iteratively. At every junction $v$ with incoming edge $e \in \delta^+(v)$ and outgoing edge $\tilde{e} \in \delta^-(v)$, which are part of the car's path, we set
\begin{align*}
	t_{\tilde{e}}^* &= at_e(t_e^*) + wt_v(at_e(t_e^*)), \\
	x_{\tilde{e}}^* &= a_{\tilde{e}}.
\end{align*}

\subsubsection{Choice of the Next Road} \label{sec:CarChoice}
Whenever a car arrives at a dispersing junction, the driver decides which outgoing road will be taken next. One possibility is the use of the road being part of the shortest path to the destination in terms of total road length. Another approach might consider travel and waiting times to obtain the fastest path. In the following, we apply Dijkstra's algorithm \cite{Dijkstra.1959}, where the definition of time-varying weights is the crucial point. 
\\
\\
\textbf{Complete Information}
\\
In a first situation, the car has complete information on the traffic and on its own properties including the initial position, the destination and the departure time. Clearly, following the shortest path generally does not correspond to the earliest arrival time at the destination. For this reason, we include the travel and waiting times defined previously to find the fastest path. In \cite{Cooke.1966}, Cook and Harley first dealt with the time-dependent shortest path problem. In general, a path can be written by a finite sequence of nodes and edges $(v_0,e_1,v_1,...,e_k,v_k)$ such that $v_0,...,v_k \in \cV$ and $e_1,...,e_k \in \cE$ with $A(e_i)=v_{i-1}$ and $\Omega(e_i)=v_i$ for $i=1,...,k$. The set of all paths from a source $s \in \cV$ to a destination $d \in \cV$ is denoted by $\cP(s,d)$. For a given path $P=(s,e_1,v_1,...,e_k,d) \in \cP(s,d)$ and a departure time $t \geq 0$, the path arrival time $at_P(t)$ is defined by the sum of the waiting time at nodes and the travel time on edges, i.e.
\begin{align*}
	at_P(t) = & t + wt_s(t) + tt_{e_1}(t+wt_s(t)) \\
	&+ wt_{v_1}(t+wt_s(t)+tt_{e_1}(t+wt_s(t))) + ... + tt_{e_k}(...).
\end{align*}
Then, our aim is to find the earliest path arrival time such that
\begin{align*}
	eat_{s,d}(t) = \min \{ at_P(t): P \in \cP(s,d) \}.
\end{align*}
The following lemma enables to apply classical methods to find the earliest arrival time path from the theory of static networks as described in \cite{Dreyfus.1969}. 
\begin{lemma}
	The network of the presented coupled model together with the definition of the total travel time on an edge 
	\begin{align}
	ttt_e(t) = wt_v(t) + tt_e(wt_v(t)),~e \in \delta^-(v). \label{eq:CP_ttt}
	\end{align} satisfies the condition of a FIFO network. 
\end{lemma}

\begin{proof}
	We have to show that for all edges $e \in \cE$ and $t < \tilde{t}$,
	\begin{align*}
	t+ttt_e(t) \leq \tilde{t} + ttt_e(\tilde{t})
	\end{align*}
	is valid. We begin by considering two cars starting at $t$ and $\tilde{t}$ with $t < \tilde{t}$. Additionally, we assume that the car which started later arrives earlier at the end of the road, that means $t + ttt_e(t) > \tilde{t} + ttt_e(\tilde{t})$. Then, this car needs to overtake the other car at some point. 
	
	While the car is inside the buffer, this is not possible since the car has to wait in an ordered queue until the cars in front of it have left the buffer. Thus, $t+wt_v(t) \leq \tilde{t}+wt_v(\tilde{t})$ and overtaking must happen on the road. 
	Then, there must exist a point $(x^o,t^o)$ with $x^o \in [a_e,b_e]$ and $t^o \geq \tilde{t}+wt_v(\tilde{t})$. At this point, the speed of both cars is equal to $\omega(\rho_e(x^o,t^o))$. Consequently, they travel at the same speed until the end of the road and arrive at the same time. We obtain $t + ttt_e(t) = \tilde{t} + ttt_e(\tilde{t})$, which leads to a contradiction.	
\end{proof}%

The consideration of FIFO networks has the advantage that the concatenation property is valid. Hence, it is sufficient to assign only one label to each node representing the earliest arrival. Then, these labels can be successively updated similar to Dijkstra's algorithm using the total travel times instead of distances between the nodes.
~\\
\\
\textbf{Departure Time Not Available}
\\
We study the situation in which we know the aggregated densities and buffer loads over time but the exact departure time of the car is unknown. Therefore, we use an aggregated measure to determine a path instead of requiring exact waiting and travel time functions. This approach has its origin in the field of optimization of traffic flow models \cite{Garavello.2016,Garavello.2006}. A possible target can be the minimization of the travel time for all cars in the network within a given time horizon as studied in \cite{Treiber.2013}.

First, we consider the total number of cars on an edge $e$ on the complete time horizon $[0,T]$, which reads
\begin{align*}
	\int_0^T \int_{I_e} \rho_e(x,t) ~dx dt.
\end{align*}
We know that the higher this quantity, the higher the total time spent by cars traveling along this road. Thus, we obtain a measure corresponding to the travel time. 
Next, we must include the waiting time at nodes in some way. The buffer load over the interval $[0,T]$ is given by
\begin{align*}
	\int_0^T r_v(t) ~dt
\end{align*}
and provides an analogy to the waiting time. To compare these two measures, we normalize the quantities by their maximum value and obtain
\begin{align*}
	\lambda^{\rho}_e &= \frac{1}{T \max_{e \in \cE} \{L_e\}} \int_0^T \int_{I_e} \rho_e(x,t) ~dx dt, \\
	\lambda^{r}_v &= \frac{1}{T \max_{v \in \cV \backslash \mathcal{\tilde{V}}} \{r^{max}_v\} } \int_0^T r_v(t) ~dt.
\end{align*}
Now, we combine these single elements and define for parameters $w^{\rho},w^{r} \in [0,1]$ with $w^{\rho} + w^{r} = 1$ the weights $\lambda_e$ on each edge $e \in \cE$,
\begin{align}
	\lambda_e &= w^{\rho} \lambda^{\rho}_e + w^{r} \lambda^{r}_v \nonumber \\ 
			  &= \frac{w^{\rho}}{T \max_{e \in \cE} \{L_e\}} \int_0^T \int_{I_e} \rho_e(x,t) ~dx dt + \frac{w^{r}}{T r^{max}} \int_0^T r_v(t) ~dt  ,~~~ e\in \delta^-(v). \label{eq:Weights}
\end{align} 
These time-independent weights enable to apply static methods, as for instance Dijkstra's algorithm, to find a path with minimum weights from the target node of the initial edge $s=\Omega(e^*)$ to the destination $d$. Instead of using the distance of each edge, we use these weights to find a path. Then, the car chooses the next edge in $P$ at every junction.
\\
\\
\textbf{Departure Time and Current Information Available}
\\
In general, having full information about the traffic is a strong assumption. Therefore, it might be reasonable to update the path of the car during its journey based on the current information. More precisely, at every dispersing junction, we restart a fastest path algorithm from the current node using only the information available at this time point. Hence, the waiting and travel times are not given as input functions, but only for the fixed and current times $\bar{t}$ when the car is at a dispersing junction. Obviously, we cannot expect to obtain the exact travel and waiting times due to the lack of information, but the updating procedure provides the opportunity to change the path when the traffic on a road is suddenly more dense than somewhere else. 

In contrast to the previous approach, we consider the total number of cars on the edge $e$ only at a certain time $\bar{t}$, i.e.
\begin{align*}
	\int_{I_e} \rho_e(x,\bar{t}) ~dx.
\end{align*}
Likewise, the buffer load at this time is $r_v(\bar{t})$. Again, we normalize and obtain the values
\begin{align*}
	\lambda^{\rho}_e(\bar{t}) = \frac{1}{\max_{e \in \cE} \{L_e\}} \int_{I_e} \rho_e(x,\bar{t}) ~dx ~~~\text{and}~~~ \lambda^{r}_v(\bar{t}) = \frac{r_v(\bar{t})}{r^{max}}.
\end{align*}
Together with parameters $w^{\rho},w^{r} \in [0,1]$ satisfying $w^{\rho} + w^{r} = 1$. Then, the weights of the Dijkstra algorithm on each edge $e \in \cE$ read as
\begin{align}
	\lambda_e(\bar{t}) = w^{\rho} \lambda^{\rho}_e(\bar{t}) + w^{r} \lambda^{r}_v(\bar{t}),~ e\in \delta^-(v). \label{eq:WeightsCurr}
\end{align}
Doing so, we can apply Dijkstra's algorithm to find a path from the current node to the destination that is followed by the car up to the next dispersing junction.

\section{Numerical Discretization and Tracking Algorithms}\label{sec:Analysis}

Before discussing the discretization of the car path trajectory, an appropriate scheme to approximate the solution of the traffic flow problem is presented. First, we introduce a spatial grid $a_e = x_{e,0} < x_{e,1} < \ldots < x_{e,N_e}=b_e$ of $N_e{+}1$, $N_e \in \N$, points on all roads of length $L_e=b_e-a_e$, $e\in \cE$ with step size $h_e = x_{e,i} - x_{e,i-1}, ~\text{for all}~i \in \{1,\ldots,N_e\}$. Additionally, we use $M{+}1$, $M \in \N$, time points $t^n$ to discretize the time interval $[0,T]$, with time horizon $T>0$, such that $0 = t^0 < t^1 < ... < t^M =T$.
To achieve numerical convergence, the temporal step size $\tau$ has to be linked to the spatial step size $h_e$ by the CFL-condition. 
On a road network, $h_e$ can vary due to different road lengths. To define coupling conditions at junctions, we need a uniform time discretization for all roads. For the flux function $f(\rho)=1-\rho$, the choice
\begin{align*}
	\tau \leq \frac{1}{2} \min_{e \in \mathcal{E}} \bigg\{ \frac{h_e}{\max_{\rho_e} |f'(\rho_e)|} \bigg\} = \frac{1}{2} \min_{e \in \cE}\{h_e\} 
\end{align*}
guarantees that waves emerging at the cell boundaries to be defined do not interact. For the discretization of the LWR equation on each road $e \in \cE$, we use the Godunov's method to approximate the cell averages 
\begin{align*}
	\rho_{e,i+0.5}^n \approx \frac{1}{h_e} \int_{x_{e,i}}^{x_{e,i+1}} \rho_e(x,t^n) ~~dx,~~~~i\in\{0,...,N_e-1\},~n\in \{0,...,M\},
\end{align*}
such that
\begin{align}	
	\rho_{e,0.5}^{n+1}   &= \rho_{e,0.5}^n - \tfrac{\tau}{h_e} \big(F^G(\rho_{e,0.5}^n,\rho_{e,1.5}^n) - q_e^{in,n}\big) \text{,} \\
	\rho_{i+0.5}^{e,n+1} &= \rho_{e,i+0.5}^n - \tfrac{\tau}{h_e} \big(F^G(\rho_{e,i+0.5}^n,\rho_{e,i+1.5}^n) - F^G(\rho_{e,i-0.5}^n,\rho_{e,i+0.5}^n) \big)\text{,} \\
	\rho_{e,N-0.5}^{n+1} &= \rho_{e,N-0.5}^n - \tfrac{\tau}{h_e} \big(q_e^{out,n} - F^G(\rho_{e,N-1.5}^n,\rho_{e,N-0.5}^n)\big) \text{,}
\end{align}
where $F^G(u,v)$ denotes the numerical flux of the Godunov scheme. The fluxes at the boundary $q_e^{in,n}$ and $q_e^{out,n}$ are obtained by replacing the continuous variables in \eqref{eq:a_dB}-\eqref{eq:c_dB} by their discrete versions. As approximation to $r_v(t^n)$, we use the first order explicit Euler method and obtain
\begin{align*}
	r_v^{n+1} = r_v^{n} + \tau~ \left(\sum_{e\in \delta^+(v)} q_e^{out,n} - \sum_{e \in \delta^-(v)} q_e^{in,n}\right).
\end{align*}
The presented approximation of the density and the buffer load allow for a discretization of the car's path. In Subsection \ref{sec:Naive}, a naive algorithm to track a car on a single road is investigated before applying a more complex one in Subsection \ref{sec:Complex}. Section \ref{sec:TravelWaiting} is aimed at approximating the travel and waiting times. For completeness, we give a brief description of the discrete versions of the edge weights to determine the next road of the car's path in Subsection \ref{chap:NextRoad}.

\subsection{Naive Tracking Algorithm}\label{sec:Naive}

An intuitive approach to solve an ordinary differential equation numerically might be the explicit Euler method. We consider a car starting at time $t^0=0$ at position $x^0=x^*$ and iteratively compute the position $x^{n+1}$ at the next time $t^{n+1}$ by 
\begin{align*}
	x^{n+1} = x^n + \tau v(\rho_{i+0.5}^n).
\end{align*}
The speed of the car is based on the average density $\rho_{i+0.5}^n$ of the cell in which the car is located at time $t^n$. Note that we drop the index $e$ here since we only consider a single road. We terminate the procedure above when the car either crosses the end of the road or the time horizon $T$ is reached. 

Apparently, we cannot completely avoid an error due to the underlying approximated cell averages. However, we expect this error to decline with decreasing step size. The naive algorithm does not adjust the car's speed when it reaches a new cell with a possibly different density within a time step. Additionally, waves are assumed to emerge at cell boundaries and to travel along the road, which should immediately influence the speed of the car when hit by a wave. Therefore, we also consider a more complex algorithm including interactions with waves.

\subsection{Complex Tracking Algorithm} \label{sec:Complex} 

The more complex algorithm is based on the work of \cite{Bretti.2008} to track a car in a classical traffic network without buffers at junctions. The idea of this algorithm is to exactly solve the coupled ordinary differential equation at every time step using the piecewise constant densities resulting from Godunov's method to describe the speed of the car. We start by defining the shifted cells 
\begin{align*}
	C_0 = [x_0,x_{0.5}), ~~~C_i = [x_{i-0.5},x_{i+0.5}),~i=1,...,N-1, ~~~ C_N = [x_{N-0.5},x_N].
\end{align*}
Since the cell averages used for the Godunov's scheme are defined from $x_i$ to $x_{i+1}$, the emerging waves start exactly in the middle of these shifted cells $C_i$. For a schematic illustration of the shifted cells we refer to Figure \ref{im:CellsCar}. 

\begin{figure}[h]
	\centering
	\begin{tikzpicture}[decoration=brace,xscale=1.5]
	\draw [->] (-0.5,0) -- (9,0);
	
	\draw[decorate]  (1.02,2) --  node[above] {$C_{i-1}$} (2.98,2);
	\draw[decorate]  (3.02,2) --  node[above] {$C_{i}$}(4.98,2);
	\draw[decorate]  (5.02,2) --  node[above] {$C_{i+1}$}(6.98,2);
	
	\draw [dashed] (0,0) -- (0,1.8);
	\draw [dashed] (2,0) -- (2,1.8);
	\draw [dashed] (4,0) -- (4,1.8);
	\draw [dashed] (6,0) -- (6,1.8);
	\draw [dashed] (8,0) -- (8,1.8);
	
	\coordinate[label=below:$x_{i-2}$] ()  at (0,0);
	\coordinate[label=below:$x_{i-1}$] ()  at (2,0);
	\coordinate[label=below:$x_{i}$] ()  at (4,0);
	\coordinate[label=below:$x_{i+1}$] ()  at (6,0);
	\coordinate[label=below:$x_{i+2}$] ()  at (8,0);
	
	\draw (0,1.1) -- node[above] {$~~\rho_{i-1.5}^n$} (2,1.1);
	\draw (2,0.7) -- node[above] {$~~\rho_{i-0.5}^n$} (4,0.7);
	\draw (4,1.2) -- node[above] {$~~\rho_{i+0.5}^n$}(6,1.2);
	\draw (6,1) -- node[above] {$~~\rho_{i+1.5}^n$} (8,1);
	
	\draw [dotted, line width=0.3mm] (1,0) -- (1,1.9);
	\draw [dotted, line width=0.3mm] (3,0) -- (3,1.9);
	\draw [dotted, line width=0.3mm] (5,0) -- (5,1.9);
	\draw [dotted, line width=0.3mm] (7,0) -- (7,1.9);
	
	\draw (0,0.05) -- (0,-0.05);
	\draw (2,0.05) -- (2,-0.05);
	\draw (4,0.05) -- (4,-0.05);
	\draw (6,0.05) -- (6,-0.05);
	\draw (8,0.05) -- (8,-0.05);
	
	\end{tikzpicture}
	\caption{Schematic illustration of the cells $C_i$ and average densities $\rho_{i+0.5}^n$}
	\label{im:CellsCar}
\end{figure}
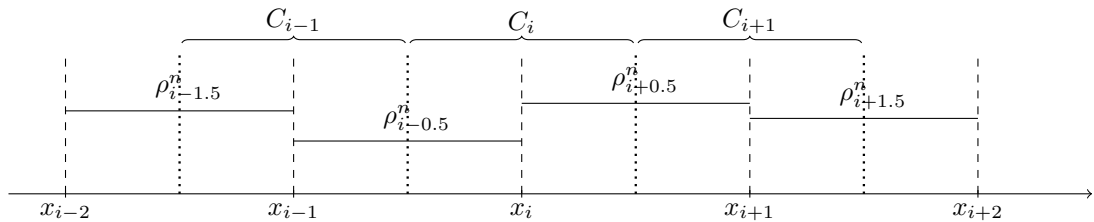

\begin{lemma} 
	Let $x^n \in C_i$ be the position of the car at time $t^n$. Further, the concave flux function $f(\rho)=\rho v(\rho)$ and cell averages approximated by the Godunov's method are assumed. Then, under the condition $\tau \leq \frac{1}{2} \frac{h}{\max |f'(\rho)|}$, for $x^n \in [x_{i-0.5},x_i)$ only the wave starting at $x_i$, and similar for $x^n \in [x_i,x_{i+0.5})$ only the wave starting at $x_{i+1}$, can influence the car trajectory within a time step $\tau$.
\end{lemma}

\begin{proof} 
By the CFL-condition $\tau \leq \frac{1}{2} \frac{h}{\max |f'(\rho)|}$, two neighboring waves do not interact before $\tau$ and the car can travel at most a distance of half a cell in one time step due to $v(\rho) \leq 1$. Consequently, for $x^n \in [x_{i-0.5},x_i]$ only a wave starting at $x_i$ can intersect the car trajectory before $\tau$.

If $x^n \in [x_i,\xr)$, we take a closer look at the wave starting at $x_i$. In case of a shock wave, i.e. $\rho_{i-0.5}<\rho_{i+0.5}$, its speed can be calculated according to the Rankine-Hugoniot-condition. Since $f(\rho) = \rho v(\rho)$ with a strictly decreasing function $v$, we can write
\begin{align*}
	\lambda_i  &= \frac{\rrp v(\rrp) - \rrm v(\rrm)}{\rrp - \rrm} < \frac{\rrp v(\rrm) - \rrm v(\rrm)}{\rrp - \rrm} = v(\rrm) = \dot{x}.
\end{align*}
In case of a rarefaction, the wave propagates at the speed
\begin{align*}
	\lambda_i = f'(\rrp) = v(\rrp) + \rrp v'(\rrp) \leq v(\rrp) = \dot{x}.
\end{align*}
The inequality holds due to the decreasing velocity, $v'<0$, and $\rho_{i+0.5}^n \geq 0$. In both cases, we see that the wave is never faster than the car. We conclude that a wave starting at $x_i$ has no intersection with the car trajectory if $x^n \in [x_i,x_{i+0.5})$.
\end{proof}
Then, the algorithm can be formulated in the following. We divide the cell $C_i = [\xl,\xr)$ in half, i.e. either
\begin{align*}
	x^n \in [\xl,x_i) ~~\text{or}~~ x^n \in [x_i,\xr),
\end{align*} 
and only study the case  $x^n \in [\xl,x_i)$ since the procedure for $x^n \in [x_i,\xr)$ is similar. Further, we distinguish between no wave, shock and rarefaction. For simplicity, we assume $f(\rho)=\rho v(\rho) = \rho(1-\rho)$, which allows for explicit representation.

When no wave occurs at $x_i$, the car is not deflected away from its original path and its speed remains constant on $[t^n,t^{n+1}]$. In particular, the new position at time $t^{n+1}$ can be calculated by the explicit Euler method, similar to the naive algorithm, i.e.
\begin{align*}
	x^{n+1} = x^n + \tau v(\rrm).
\end{align*}
If $\rrm < \rrp$, the path of the car possibly consists of two parts with different velocities. Starting with $v(\rrm)$, the car drives at a lower speed $v(\rrp)$ after an interaction with the shock front. The cell $C_i$ together with the car trajectory and the shock front are depicted in Figure \ref{im:Shock}.

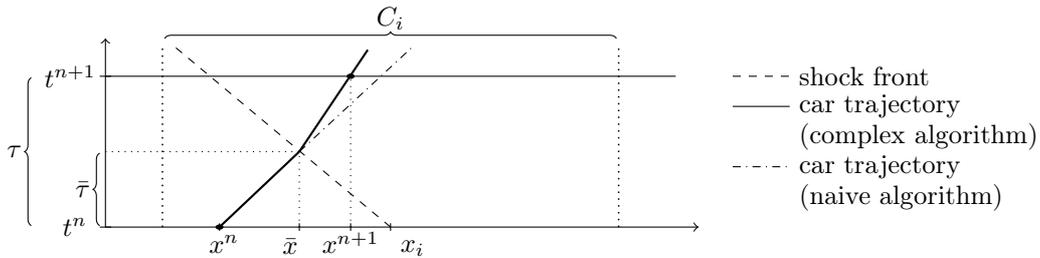
\begin{figure}[h]
	\centering
	\begin{tikzpicture}[decoration=brace,xscale=1.5]
	
	\draw [->] (0,0) -- (5.2,0);
	\draw [->] (0,0) -- (0,2.5);
	\draw (0,2) -- (5,2);
	
	\coordinate[label=left:$t^n~$] ()  at (0,0);
	\draw (-0.05,0) -- (0.05,0);
	\coordinate[label=left:$t^{n+1}$] ()  at (0,2);
	\draw (-0.05,2) -- (0.05,2);
	
	\draw [decorate] (-0.65,0.02) -- node[left] {$\tau$}(-0.65,1.98);
	
	\draw [decorate] (-0.05,0.02) -- node[left] {$\bar{\tau}$}(-0.05,0.98);
	
	\draw [dashed] (2.5,0) -- (0.6,2.4);
	\coordinate[label=below:$~~~~~x_{i}$] ()  at (2.5,-0.06);
	\draw (2.5,0.05) -- (2.5,-0.05);

	\draw[decorate]  (0.52,2.5) --  node[above] {$C_{i}$}(4.48,2.5);
	\draw [dotted, line width=0.2mm] (0.5,0) -- (0.5,2.4);
	\draw [dotted, line width=0.2mm] (4.5,0) -- (4.5,2.4);
	
	\draw [thick] (1,0) -- (1.7,1);
	\draw [dash dot] (1.7,1) -- (2.7,2.4);
	\draw (1,0.05) -- (1,-0.05);
	\coordinate [label=below:$~x^{\scriptsize{n}}$] ()  at (1,0.02);
	\fill[black] (1,0) circle (1pt);
	
	\draw [dotted] (1.7,0) -- (1.7,1);
	\draw [dotted] (0,1) -- (1.7,1);
	
	\draw [dotted] (2.15,0) -- (2.15,2);
	\draw (2.15,0.05) -- (2.15,-0.05);
	\coordinate[label=below:$x^{\scriptsize{n+1}}$] ()  at (2.15,0.065);
	\fill[black] (2.15,2) circle (1pt);
	
	\coordinate[label=below:$\bar{x}~~$] ()  at (1.7,-0.03);
	\draw (1.7,0.05) -- (1.7,-0.05);
	
	\draw [thick] (1.7,1) -- (2.3,2.35);
	
	\begin{scope}[shift={(5.5,1.2)}]
		\path [yshift=-0.8cm](0,0) -- plot[] (0.25,0) -- (0.5,0) node[right]{(naive algorithm)};
		\draw [dash dot,yshift=-0.4cm](0,0) -- plot[] (0.25,0) -- (0.5,0) node[right]{car trajectory};
		\path [](0,0) -- plot[] (0.25,0) -- (0.5,0) node[right]{(complex algorithm)};
		\draw [yshift=0.4cm](0,0) -- plot[] (0.25,0) -- (0.5,0) node[right]{car trajectory};
		\draw [dashed,yshift=0.8cm](0,0) -- plot[] (0.25,0) -- (0.5,0) node[right]{shock front};
	\end{scope}
	\end{tikzpicture}
	\caption{Trajectory of a car that hits a shock front at the intersection $(t^n+\bar{\tau},\bar{x})$}
	\label{im:Shock}
\end{figure}

Until the time $t^n+\bar{\tau}$, when the car interacts with the wave, the car trajectory can be expressed by
\begin{align*}
	x(t) = x^n + (t-t^n) v(\rrm) ,~~t \in [t^n,t^n + \bar{\tau}].
\end{align*}
The velocity of a shock wave starting at $x_i$ at time $t^n$ satisfies
\begin{align*}
	\lambda_i = \frac{f(\rrp) - f(\rrm)}{\rrp - \rrm}.
\end{align*}
When the car hits the shock front, which reads
\begin{align}
	x^n + v(\rrm) \bar{\tau} = x_i + \lambda_i \bar{\tau},
\end{align} 
the point of the intersection $(t^n+\bar{\tau},\bar{x})$ is given by 
\begin{align}
	\bar{\tau} = \frac{x_i - x^n}{v(\rrm) - \lambda_i} \text{~~~~and~~~~}
	\bar{x} = x^n + v(\rrm) \bar{\tau}. \label{eq:barX}
\end{align}
We note that $\bar{\tau}$ is well-defined since it can be shown that $v(\rrm) \neq \lambda_i$ by employing $f(\rho)=\rho v(\rho) = \rho(1-\rho)$. Additionally, $\bar{\tau} > 0 $ due to $x^n \in  [\xl,x_i)$. Now, we distinguish: 
\begin{enumerate}[(i)]
	\item $\bar{\tau} \geq \tau$, i.e., the shock and the car do not interact before $\tau$ and the car's position at time $t^{n+1}$ is
	\begin{align*}
	x^{n+1} = x^n + v(\rrm) \tau.
	\end{align*}
	\item $\bar{\tau} < \tau$, i.e., we have an intersection at $(t^n+\bar{\tau},\bar{x})$ and the new position at time $t^{n+1}$ is
	\begin{align*}
	x^{n+1} = \bar{x} + v(\rrp) (\tau -\bar{\tau}).
	\end{align*}
\end{enumerate}

If $\rrm > \rrp$, a rarefaction starts at $x_i$ at time $t^n$ and we need to consider three different elements of the car trajectory. First, the car will drive at speed $v(\rrm)$. If it comes to an intersection with the rarefaction, the speed continually increases until the car possibly leaves the rarefaction. Figure \ref{im:RareCar} shows a schematic illustration of the trajectory.

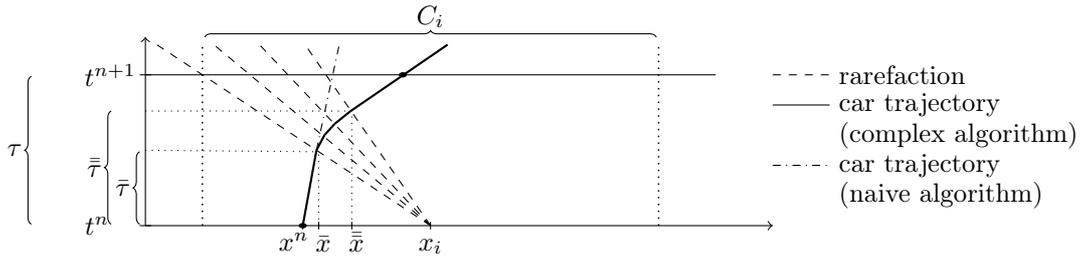
\begin{figure}[h]
	\centering
	\begin{tikzpicture}[decoration=brace,xscale=1.5]
	
	\draw [->] (0,0) -- (5.5,0);
	\draw [->] (0,0) -- (0,2.5);
	\draw (0,2) -- (5,2);
	
	\coordinate[label=left:$t^n~~~$] ()  at (0,0);
	\draw (-0.05,0) -- (0.05,0);
	\coordinate[label=left:$t^{n+1}$] ()  at (0,2);
	\draw (-0.05,2) -- (0.05,2);
	\fill[black] (2.257,2) circle (1pt);
	
	\draw [decorate] (-1,0.02) -- node[left] {$\tau$}(-1,1.98);
	\draw [decorate] (-0.3,0.02) -- node[left] {$\bar{\bar{\tau}}$}(-0.3,1.52);
	\draw [decorate] (-0.05,0.02) -- node[left] {$\bar{\tau}$}(-0.05,0.98);
	
	\draw [dashed] (2.5,0) -- (0.1,2.4);
	\draw [dashed] (2.5,0) -- (0.6,2.4);
	\draw [dashed] (2.5,0) -- (1,2.4);
	\draw [dashed] (2.5,0) -- (1.4,2.4);
	
	\draw [dash dot] (1.5,1) -- (1.7,2.4);
	
	\coordinate[label=below:$x_{i}$] ()  at (2.5,-0.05);
	\draw (2.5,0.05) -- (2.5,-0.05);
	
	\draw[decorate]  (0.52,2.5) --  node[above] {$C_{i}$}(4.48,2.5);
	\draw [dotted, line width=0.2mm] (0.5,0) -- (0.5,2.4);
	\draw [dotted, line width=0.2mm] (4.5,0) -- (4.5,2.4);
	
	\draw [thick] (1.38,0) -- (1.5,1);
	\draw [thick] (1.5,1) -- (1.57,1.2);
	\draw [thick] (1.57,1.2) -- (1.66,1.35);
	\draw [thick] (1.66,1.35) -- (1.8,1.52);
	\draw [thick] (1.8,1.52) -- (2.65,2.4);
	
	\fill[black] (1.38,0) circle (1pt);
	\coordinate [label=below:$x^n~~$] ()  at (1.38,0.03);
	\draw (1.38,0.05) -- (1.38,-0.05);
	
	\draw [dotted] (1.52,0) -- (1.52,1);
	\draw [dotted] (0,1) -- (1.52,0.98);
	
	\draw [dotted] (1.81,0) -- (1.81,1.52);
	\draw [dotted] (0,1.52) -- (1.81,1.52);
	
	\coordinate[label=below:$~\bar{x}$] ()  at (1.52,-0.02);
	\draw (1.52,0.05) -- (1.52,-0.05);
	
	\coordinate[label=below:$~\bar{\bar{x}}$] ()  at (1.81,0.03);
	\draw (1.81,0.05) -- (1.81,-0.05);
	
	\begin{scope}[shift={(5.5,1.2)}]
		\path [yshift=-0.8cm](0,0) -- plot[] (0.25,0) -- (0.5,0) node[right]{(naive algorithm)};
		\draw [dash dot,yshift=-0.4cm](0,0) -- plot[] (0.25,0) -- (0.5,0) node[right]{car trajectory};
		\path [](0,0) -- plot[] (0.25,0) -- (0.5,0) node[right]{(complex algorithm)};
		\draw [yshift=0.4cm](0,0) -- plot[] (0.25,0) -- (0.5,0) node[right]{car trajectory};
		\draw [dashed,yshift=0.8cm](0,0) -- plot[] (0.25,0) -- (0.5,0) node[right]{rarefaction};
	\end{scope}
	\end{tikzpicture}
	\caption{Trajectory of a car that reaches a rarefaction at $(t^n+\bar{\tau},\bar{x})$ and leaves at $(t^n+\bar{\bar{\tau}},\bar{\bar{x}})$ }
	\label{im:RareCar}
\end{figure}

With the choice $v(\rho) = 1-\rho$, the ODE to be solved in case of a rarefaction starting at $(t^n,x_i)$ becomes
\begin{align*}
	\dot{x}(t) &= v(\rho(x(t),t)) = 1 - \rho(x(t),t) \\
			   &= \begin{cases}
						1-\rrm & \text{if } ~ x(t) \leq x_i + f'(\rrm)(t-t^n),\\
						\frac{1}{2} + \frac{x(t)-x_i}{2(t-t^n)} & \text{if } ~ f'(\rrm) (t-t^n) < x(t)-x_i < f'(\rrp) (t-t^n), \\
						1-\rrp & \text{if } ~ x(t) \geq x_i + f'(\rrp)(t-t^n).
				  \end{cases}
\end{align*}
The first intersection is denoted again by $(t^n+\bar{\tau},\bar{x})$ can be derived in the same manner as in the case of a shock. Instead of applying the 
Rankine-Hugoniot-condition to obtain the speed of the wave, we take the velocity at which the rarefaction propagates to the left, i.e.
\begin{align*}
	\lambda_i = f'(\rrm).
\end{align*}
Thus, we can compute the coordinates of this intersection according to 
\begin{align*}
	\bar{\tau} = \frac{x_i - x^n}{v(\rrm) - \lambda_i} \text{~~~~and~~~~} \bar{x} = x^n + v(\rrm) \bar{\tau}.
\end{align*}
Again, it can be shown that $\bar{\tau}$ is well-defined and always positive. To describe the car trajectory while the car is in the rarefaction, we have to solve the inhomogeneous linear ordinary differential equation
\begin{align*}
	\dot{x}(t) = \frac{1}{2} + \frac{x(t)-x_i}{2(t-t^n)}.
\end{align*}
Applying the method variation of parameters, the solution reads
\begin{align}
	x(t) = x_i + t - t^n - \sqrt{t-t^n} ~ \frac{\bar{\tau} + x_i - \bar{x}}{\sqrt{\bar{\tau}}},~~~t \in [t^n+\bar{\tau},t^n+\bar{\bar{\tau}}]. \label{eq:VarConst}
\end{align}
Note that $t^n+\bar{\bar{\tau}}$ defines the time at which the car leaves the rarefaction. However, $\bar{\bar{\tau}}$ is not necessarily smaller than infinity. It is also possible that the car never leaves the rarefaction. To find out if a final intersection at $\bar{\bar{x}}$ at time $t^n+\bar{\bar{\tau}}$ exists, compare Figure \ref{im:RareCar}, we solve the equation
\begin{align}
	x(t^n+\bar{\bar{\tau}}) = x_i + f'(\rrp) \bar{\bar{\tau}}, \label{eq:xBarBar}
\end{align}
where $x(t^n+\bar{\bar{\tau}})$ is given by \eqref{eq:VarConst}. Rearranging leads to
\begin{align}
	(1-f'(\rrp)) ~\bar{\bar{\tau}} - \frac{\bar{\tau}+x_i-\bar{x}} {\sqrt{\bar{\tau}}} ~\sqrt{\bar{\bar{\tau}}} = 0. \label{eq:NonlinEq}
\end{align}
It can be shown that $\frac{\bar{\tau}+x_i-\bar{x}} {\sqrt{\bar{\tau}}} \neq 0$ by using the explicit values for $\bar{\tau}$ and $\bar{x}$. If $\rrp = 0$, the unique solution to this non-linear equation is $\bar{\bar{\tau}} = 0$. This means, that there is no final intersection point since the car is not fast enough to leave the rarefaction. To be more precise, the rarefaction propagates at speed $f'(\rrp)=1$ to the right. At the same time, the car becomes faster while driving through the rarefaction. However, it can reach at most the speed $v_{max} = 1$. Hence, it cannot overtake the right rarefaction front. Otherwise, if $\rrp \neq 0$, \eqref{eq:NonlinEq} has two solutions 
\begin{align*}
	\bar{\bar{\tau}}_1 = 0 \text{~~~~and~~~~} \bar{\bar{\tau}}_2 = \left(\frac{\bar{\tau}+x_i-\bar{x}} {\sqrt{\bar{\tau}}~(1-f'(\rrp))}\right)^2.
\end{align*}
The second candidate is the desired $\bar{\bar{\tau}}$, which is well-defined since $\rrp \neq 0$ and $\bar{\tau} > 0$. By \eqref{eq:VarConst}-\eqref{eq:xBarBar}, the position of the final intersection $\bar{\bar{x}}$ is given by 
\begin{align*}
	\bar{\bar{x}} = x(t^n + \bar{\bar{\tau}})  = x_i + \bar{\bar{\tau}}  - \sqrt{\bar{\bar{\tau}}} \frac{\bar{\tau} + x_i - \bar{x}}{\sqrt{\bar{\tau}}} = x_i + f'(\rrp) \bar{\bar{\tau}}. 
\end{align*}
To calculate the new position of the car, we differentiate the following cases.
\begin{enumerate}[(i)]
	\item $\bar{\tau} \geq \tau$, i.e. no interaction takes place before $\tau$ and the new position of the car at time $t^n$ is given by
	\begin{align*}
		x^{n+1} = x^n + v(\rrm) \tau.
	\end{align*}
	\item $\bar{\tau} < \tau$, i.e. the car interacts with the rarefaction before $\tau$ and we need to further distinguish: 
	\begin{enumerate}
		\item $\bar{\bar{\tau}} \geq \tau$, i.e. the car is still in the rarefaction at time $t^{n+1}$. Applying \eqref{eq:VarConst} yields
		\begin{align*}
			x^{n+1} = x_i + \tau - \sqrt{\tau} ~ \frac{\bar{\tau} + x_i - \bar{x}}{\sqrt{\bar{\tau}}}.
		\end{align*}
		\item $\bar{\bar{\tau}} <\tau$, i.e. the car has left the rarefaction before $\tau$ and we also have to include the part of the trajectory when the car drives at speed $v(\rrp)$. We obtain
		\begin{align*}
			x^{n+1} = \bar{\bar{x}} + v(\rrp) (\tau - \bar{\bar{\tau}}).
		\end{align*}
	\end{enumerate}
\end{enumerate}
According to the described procedure, we would have to consider the wave starting at $x_N$ for a car located in the last cell $C_N = [x_{N-0.5},x_N]$. Since $x_N = b$ is the endpoint of the road, we do not have a density $\rho_{N+0.5}^n$. A possible remedy is to extend the road artificially assuming the same density $\rho_{N-0.5}^n$ as in the last cell, see Figure \ref{im:endOfRoad}. This implies that no wave emerges from $x_N$ and the car drives at constant speed $v(\rho_{N-0.5}^n)$ to the end of the road. 

\begin{figure}[h]
	\centering
	\begin{tikzpicture}[decoration=brace,xscale=1.5]
	\draw [->] (-0.5,0) -- (8,0);
	
	\draw[decorate]  (1.02,2) --  node[above] {$C_{N-2}$} (2.98,2);
	\draw[decorate]  (3.02,2) --  node[above] {$C_{N-1}$}(4.98,2);
	\draw[decorate]  (5.02,2) --  node[above] {$C_{N}$}(5.98,2);
	
	\draw [dashed] (0,0) -- (0,1.8);
	\draw [dashed] (2,0) -- (2,1.8);
	\draw [dashed] (4,0) -- (4,1.8);
	\draw [dashed] (6,0) -- (6,1.8);
	
	\coordinate[label=below:$x_{N-3}$] ()  at (0,0);
	\coordinate[label=below:$x_{N-2}$] ()  at (2,0);
	\coordinate[label=below:$x_{N-1}$] ()  at (4,0);
	\coordinate[label=below:$x_{N}$] ()  at (6,0);
	
	\draw (0,1.1) -- node[above] {$~~\rho_{N-2.5}^n$} (2,1.1);
	\draw (2,0.7) -- node[above] {$~~\rho_{N-1.5}^n$} (4,0.7);
	\draw (4,1.2) -- node[above] {$~~\rho_{N-0.5}^n$}(6,1.2);
	\draw [dash dot](6,1.2) -- node[above] {$~~\rho_{N+0.5}^n = \rho_{N-0.5}^n$} (7.9,1.2);
	
	\draw [dotted, line width=0.2mm] (1,0) -- (1,1.9);
	\draw [dotted, line width=0.2mm] (3,0) -- (3,1.9);
	\draw [dotted, line width=0.2mm] (5,0) -- (5,1.9);
	\draw [dotted, line width=0.2mm] (6,0) -- (6,1.9);
	\end{tikzpicture}
	\caption{Density at the end of a road}
	\label{im:endOfRoad}
\end{figure}
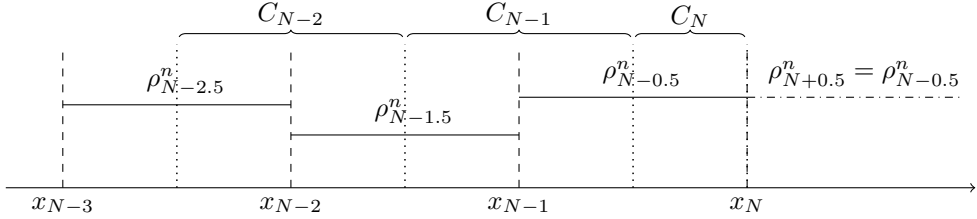

\subsection{Travel and Waiting Time} \label{sec:TravelWaiting}

For both approaches, i.e. the naive and complex algorithm, we need to discretize the behavior at nodes. Let the car be located at time $t^{n^*} \in \{t^0,...,t^M\}$ at $x_e^*$ on an edge $e \in \mathcal{E}$. For a sufficiently large time horizon $T$, our presented tracking algorithms terminate with $n \in \{n^*,...,M\}$ such that condition $x^n \geq b_e$ is valid. For $x^n > b_e$, the car has driven too far and an intermediate step to calculate the arrival time exactly at the point $b_e$ is needed. 

Since no wave emerges from $b_e$, the car's speed can be assumed to be constant and we are looking for the smallest $\hat{n} \in \{n^*,...,M\}$ and appropriate $\hat{\tau} \in [0,\tau)$ such that
\begin{align}
	x^{\hat{n}} + \hat{\tau}  v( \rho_{e,N_e-0.5}^{\hat{n}} ) = b_e ~~~
	\Leftrightarrow ~~~\hat{\tau} = \frac{b_e-x^{\hat{n}}}{v( \rho_{e,N_e-0.5}^{\hat{n}} )}, \label{eq:that}
\end{align}
which means that at time $t^{\hat{n}}+\hat{\tau}$, the car is located at $b_e$. Note that $v( \rho_{e,N_e-0.5}^{\hat{n}} ) \neq 0$, otherwise the car is waiting and must have already reached the end of the road in the time step before. Hence, $\hat{\tau}$ is well-defined and the travel time $tt_e(t_e^{n^*})$ on road $e$ can be approximated by 
\begin{align*}
	tt_e^{n^*} = (\hat{n}-n^*)\tau + \hat{\tau}.
\end{align*}
Since we know the arrival time $t^{\hat{n}}+\hat{\tau}$ at the end of road $e$, we are able to approximate the buffer load $r_v(t^{\hat{n}}+\hat{\tau})$ by $\hat{r}_v$ with $e \in \delta^+(v)$. For this purpose, we use linear interpolation between $t^{\hat{n}}$ and $t^{\hat{n}+1}$ and define 
\begin{align*}
	\hat{r}_v = r^{\hat{n}}_v + \hat{\tau} (f^{in,\hat{n}}_v - f^{out,\hat{n}}_v),
\end{align*}
with the total in- and outflow at the node $v$ at time $t^{\hat{n}}$ denoted by $f^{in,\hat{n}}_v$ and $f^{out,\hat{n}}_v$.
Similar to the continuous formulation, we distinguish two cases $\hat{r}_v = 0$ and $\hat{r}_v > 0$.

If $\hat{r}_v=0$, the car immediately passes the node and continues on the next road $\tilde{e} \in \delta^-(v)$  at a speed depending on the density in the first cell. That means, the waiting time is zero and the position of the car at the next time $t^{\hat{n}+1}$ is
\begin{align*}
	x^{\hat{n}+1} = a_{\tilde{e}} + (\tau - \hat{\tau}) v(\rho_{\tilde{e},0.5}^{\hat{n}}).
\end{align*}
This case is illustrated in the left picture of Figure \ref{im:Node}.

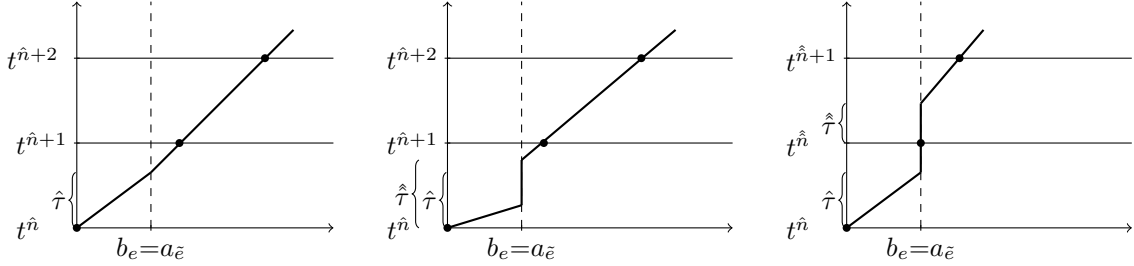
\begin{figure}[h]
	\begin{tikzpicture}[decoration={markings,mark=at position 1 with {\arrow[scale=1,black]{latex}};},scale=0.75]
	\draw [->] (1.5,0) -- (1.5,4);
	\draw [->] (1.5,0) -- (6,0);
	
	\draw [dashed] (2.8,0) -- (2.8,4);
	\draw (2.8,0.05) -- (2.8,-0.05);
	\coordinate[label=below:$b_e {=} a_{\tilde{e}}$] ()  at (2.8,0);
	
	\draw (1.45,0) -- (1.55,0);
	\coordinate[label=left:$t^{\hat{n}}~~~$] ()  at (1.5,0);
	
	\draw (1.5,1.5) -- (6,1.5);
	\draw (1.45,1.5) -- (1.55,1.5);
	\coordinate[label=left:$t^{\hat{n}+1}$] ()  at (1.5,1.5);
	
	\draw (1.5,3) -- (6,3);
	\draw (1.45,3) -- (1.55,3);
	\coordinate[label=left:$t^{\hat{n}+2}~$] ()  at (1.5,3);
	
	\draw [thick] (1.5,0) -- (2.8,0.98);
	\draw [thick] (2.8,0.98) -- (5.3,3.5);
	
	\fill[black] (1.5,0) circle (2pt);
	\fill[black] (3.3,1.5) circle (2pt);
	\fill[black] (4.8,3) circle (2pt);
	
	\draw [decorate,decoration ={brace}] (1.48,0) -- node[left] {$\hat{\tau}$}(1.48,0.98);

	\draw [->] (7+1,0) -- (7+1,4);
	\draw [->] (7+1,0) -- (7+6,0);
	
	\draw [dashed] (7+2.3,0) -- (7+2.3,4);
	\draw (7+2.3,0.05) -- (7+2.3,-0.05);
	\coordinate[label=below:$b_e {=} a_{\tilde{e}}$] ()  at (7+2.3,0);
	
	\draw (7+0.95,0) -- (7+1.05,0);
	\coordinate[label=left:$t^{\hat{n}}~~~$] ()  at (7+1,0);
	
	\draw (7+1,1.5) -- (7+6,1.5);
	\draw (7+0.95,1.5) -- (7+1.05,1.5);
	\coordinate[label=left:$t^{\hat{n}+1}$] ()  at (7+1,1.5);
	
	\draw (7+1,3) -- (7+6,3);
	\draw (7+0.95,3) -- (7+1.05,3);
	\coordinate[label=left:$t^{\hat{n}+2}$] ()  at (7+1,3);
	
	\draw [thick] (7+1,0) -- (7+2.3,0.4);
	\draw [thick] (7+2.3,0.4) -- (7+2.3,1.2);
	\draw [thick] (7+2.3,1.2) -- (7+5,3.5);
	
	\fill[black] (7+1,0) circle (2pt);
	\fill[black] (7+2.69,1.5) circle (2pt);
	\fill[black] (7+4.4,3) circle (2pt);
	
	\draw [decorate,decoration ={brace}] (7+0.98,0) -- node[left] {$\hat{\tau}$}(7+0.98,0.98);
	\draw [decorate,decoration ={brace}] (7+0.5,0) -- node[left] {$\doublehat{\tau}$}(7+0.5,1.2);	
	
	\draw [->] (14+1,0) -- (14+1,4);
	\draw [->] (14+1,0) -- (14+6,0);
	
	\draw [dashed] (14+2.3,0) -- (14+2.3,4);
	\draw (14+2.3,0.05) -- (14+2.3,-0.05);
	\coordinate[label=below:$b_e {=} a_{\tilde{e}}$] ()  at (14+2.3,0);
	
	\draw (14+0.95,0) -- (14+1.05,0);
	\coordinate[label=left:$t^{\hat{n}}~~~$] ()  at (14+1,0);
	
	\draw (14+1,1.5) -- (14+6,1.5);
	\draw (14+0.95,1.5) -- (14+1.05,1.5);
	\coordinate[label=left:$t^{\doublehatexp{n}}~~~$] ()  at (14+1,1.5);
	
	\draw (14+1,3) -- (14+6,3);
	\draw (14+0.95,3) -- (14+1.05,3);
	\coordinate[label=left:$t^{\doublehatexp{n}+1}$] ()  at (14+1,3);
	
	\draw [thick] (14+1,0) -- (14+2.3,0.98);
	\draw [thick](14+2.3,0.98) -- (14+2.3,2.2);
	\draw [thick](14+2.3,2.2) -- (14+3.4,3.5);
	
	\fill[black] (14+1,0) circle (2pt);
	\fill[black] (14+2.3,1.5) circle (2pt);
	\fill[black] (14+2.98,3) circle (2pt);
	
	\draw [decorate,decoration ={brace}] (14+0.98,0) -- node[left] {$\hat{\tau}$}(14+0.98,0.98);
	\draw [decorate,decoration ={brace}] (14+0.98,1.5) -- node[left] {$\doublehat{\tau}$}(14+0.98,2.2);
	
	\end{tikzpicture}
	\caption{Trajectory of a car when the car can leave immediately (left), in the same time step (center), in the next time step (right)}
	\label{im:Node}
\end{figure}

Otherwise, if $\hat{r}_v>0$, we take a closer look at the discrete outflow $f^{out,n}_v$ with $n \geq \hat{n}$. To obtain the outflow within two time steps, we use again linear interpolation. For a small buffer load and a sufficiently large outflow $f_v^{out,\hat{n}}$, it is possible that the car enters and leaves the buffer within $\tau$, as depicted in the center of Figure \ref{im:Node}. Then, the leaving time $t^{\hat{n}}+\doublehat{\tau}$ with $ \doublehat{\tau} \in (\hat{\tau},\tau)$ is given by
\begin{align*}
	\hat{r}_v = (\doublehat{\tau} - \hat{\tau}) f^{out,\hat{n}}_v ~~~\Leftrightarrow~~~ \doublehat{\tau} = \frac{\hat{r}_v +\hat{\tau} f_v^{out,\hat{n}}}{f_v^{out,\hat{n}}}.
\end{align*}
As a consequence, the approximated waiting time at the node is $wt_v^{\hat{n}}=\doublehat{\tau}-\hat{\tau}$ and the car's position at time $t^{\hat{n}+1}$ can be computed by
\begin{align*}
	x^{\hat{n}+1} = a_{\tilde{e}} + (\tau - \doublehat{\tau}) v(\rho_{\tilde{e},0.5}^{\hat{n}}).
\end{align*}
If the car has to wait longer inside the buffer, we  
look for the smallest $\doublehat{n} \in \{ \hat{n}+1, \hat{n}+2,...,M \}$ and $\doublehat{\tau} \in [0,\tau)$ such that
\begin{align*}
	\hat{r}_v = (\tau - \hat{\tau}) f^{out,\hat{n}}_v + \tau \sum_{n=\hat{n}+1}^{\doublehatexp{n}-1} f^{out,n}_v + \doublehat{\tau} f^{out,\doublehatexp{n}}_v.
\end{align*}
The right picture of Figure \ref{im:Node} shows a schematic illustration of this situation. Then, the approximated waiting time can be computed by
\begin{align*}
	wt_v^{\hat{n}} = (\tau-\hat{\tau}) + \tau \left( \doublehat{n} - 1 -(\hat{n}+1) + 1 \right) + \doublehat{\tau} = -\hat{\tau} + \tau \left( \doublehat{n} - \hat{n} \right) + \doublehat{\tau}.
\end{align*}
For all $n=\hat{n}+1,...,\doublehat{n}$, the car is located at the end of the road and its position is 
\begin{align*}
	x^n = b_e = a_{\tilde{e}}.
\end{align*}
The first position on the new edge $x^{\doublehatexp{n}+1}$ at the time $t^{\doublehatexp{n}+1}$ is then 
\begin{align*}
	x^{\doublehatexp{n}+1} = a_{\tilde{e}} + (\tau - \doublehat{\tau}) v(\rho_{\tilde{e},0.5}^{\doublehatexp{n}}).
\end{align*}
Given the first position on the new edge, we can continue with one of our tracking algorithms using this point as initial datum.

\subsection{Choice of the Next Road} \label{chap:NextRoad}

If the exact departure time is known, the car can drive along the fastest path calculated by an extended Dijkstra Algorithm. Instead of requiring continuous input functions, which provide the travel and waiting times, we can integrate the calculation of these quantities into the algorithm. We use our tracking algorithms and the procedure described in the previous section to determine the waiting time as well as the travel time on the outgoing edges of already considered nodes. Then, we possibly update the labels in the same manner as in a static Dijkstra algorithm and obtain iteratively the path with the earliest arrival time to the destination.

To determine the aggregated path, the standard MATLAB function shortestpath\footnotemark[2] can be applied using the discrete weights corresponding to \eqref{eq:Weights}
\begin{align}
	\lambda_e &= w^{\rho} \lambda^{\rho}_e + w^{r} \lambda^{r}_v \nonumber \\
			  &=  w^{\rho} \frac{\tau h}{T \max_e\{L_e\}} \sum_{n=0}^{M} \sum_{i=0}^{N_e-1} \rho_{e,i+0.5}^n + w^{r} \frac{\tau}{T r^{max}} \sum_{n=0}^{M} r_v^n,~ \text{with}~ e \in \delta^-(v),~ e\in \cE. \label{eq:WeightsD1}
\end{align}
instead of distances. 
\footnotetext[2]{Documentation: \href{https://de.mathworks.com/help/matlab/ref/graph.shortestpath.html}{https://de.mathworks.com/help/matlab/ref/graph.shortestpath.html}. Last checked: May 22, 2019. MATLAB Version R2017a.}
\\
\\
In the last considered scenario, the current traffic densities are only received at dispersing junctions. In this case, the car starts a new shortest path calculation at every junction with more than one outgoing road using the discrete weights
\begin{align}
	\lambda_e^{\hat{n}} = w^{\rho} \lambda^{\rho}_e(t^{\hat{n}}) + w^{r} \lambda^{r}_v(t^{\hat{n}}) 
	= w^{\rho} \frac{h}{\max_e \{L_e\}} \sum_{i=0}^{N_e-1} \rho_{e,i+0.5}^{\hat{n}}  +  w^{r} \frac{r_v^{\hat{n}}}{r^{max}}, \label{eq:WeightD2}
\end{align}
depending on the time point $t^{\hat{n}}$ defined in \eqref{eq:that}, when the car is almost at the end of the considered road.

Summarizing, we have the following steps. First, the driver's trajectory is solved numerically by one of the presented tracking algorithms until the car reaches the end of the road. While the destination is not reached, the waiting time at the target node of the current edge is computed and an outgoing edge is chosen according to the considered decision criteria. If it says that the car should follow the shortest, the aggregated or the fastest path, we provide these paths already as input variables since they can be calculated in advance. Then, the car simply chooses the next edge according to the appropriate path. Only in the scenario with incomplete information on the densities, the car updates its path at dispersing junctions since the densities are only available at certain time points. Then, we proceed until the destination node is reached.

\section{Computational Results} \label{sec:Results}

This section deals with the numerical evaluation of the presented algorithms based on different types of networks. We start with linear networks to show the numerical behavior of the developed algorithms at nodes. Subsequently, the tracking algorithms are applied to compare different paths in a small network. 
Finally, a car is tracked in a large network of block structure while choosing its path based on different levels of information.

\subsection{Behavior at Nodes in Linear Networks} \label{chap:BehNodesN}

In this section, we investigate the behavior at nodes for different linear networks. First, we choose a small example such that the car faces constant traffic densities on the roads and the Godunov's method provides the exact solution to the traffic flow problem. Hence, the car drives at constant speed $\bar{v}$ and our numerical tracking algorithms calculate the position on the road exactly since the updating rule
\begin{align*}
	x^{n+1} = x^n + \tau \bar{v}
\end{align*}
provides the true values. This has the advantage that arising errors must be related to the behavior at nodes. We set the parameters $L_1 = L_2 = L_3 = 1,~ r^{max}_1 = r^{max}_4 = \infty$, $r^{max}_2 = r^{max}_3 = 0.3$ and $\mu_{1} = \mu_{2} = \mu_{3} = 0.25$. The network together with the initial data are indicated in Figure \ref{im:Nodes}.

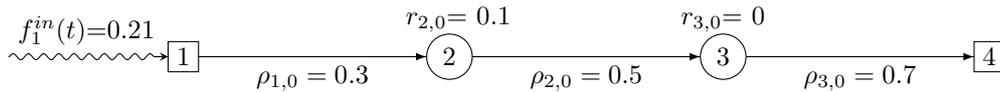
\begin{figure}[H]
	\centering
	\begin{tikzpicture}[
	decoration={
		markings,
		mark=at position 1 with {\arrow[scale=1,black]{latex}};
	}
	]
	\draw (-0.4,-0.2) rectangle (0,0.2) node [pos=.5] {1};
	
	\draw [postaction={decorate}] (0,0) -- (3,0) node [midway,below] {$\rho_{1,0} = 0.3$};
	\draw  (3.3,0) circle (3mm) node {2};
	\draw [postaction={decorate}] (3.6,0) -- (6.6,0) node [midway,below] {$\rho_{2,0} = 0.5$};
	\draw (6.9,0) circle (3mm) node {3};
	\draw [postaction={decorate}] (7.2,0) -- (10.2,0) node [midway,below] {$\rho_{3,0} = 0.7$};
	\draw (10.2,-0.2) rectangle (10.6,0.2) node [pos=.5] {4};
	
	\coordinate[label=above:$~~r_{2,0} {=} ~0.1$] ()  at (3.3,0.2);
	
	\coordinate[label=above:$r_{3,0} {=} ~0$] ()  at (6.9,0.2);

	\draw [-stealth,decoration={snake, amplitude = .4mm, segment length = 2mm, post length=0.9mm},decorate] (-2.5,0) -- (-0.4,0);
	\coordinate[label=right:$f_1^{in}(t) {=} 0.21$] ()  at (-2.5,0.35);
	\end{tikzpicture}
	\caption{Linear network with initial data}
	\label{im:Nodes}
\end{figure}

The car starts at time $t_1^*=0$ at the beginning of the first road $x_1^*=0$. Although the densities on the roads are different, no wave occurs at a junction at the beginning of the simulation. This is due to the balance property of the buffers. In Figure \ref{dia:Characteristics}, the characteristics together with the car trajectory as well as the buffer loads are depicted. At the second junction, the buffer decreases until it is completely empty. Then, a shock with positive speed emerges. However, the car is not affected since it has already passed the node at this time. The buffer at the third junction increases but does not reach $r^{max}_3 = 0.3$ within $T=8$. 

\begin{figure}[h]
	\centering
	\begin{minipage}[b]{0.48\textwidth}
		\centering
		\includegraphics[width=\textwidth]{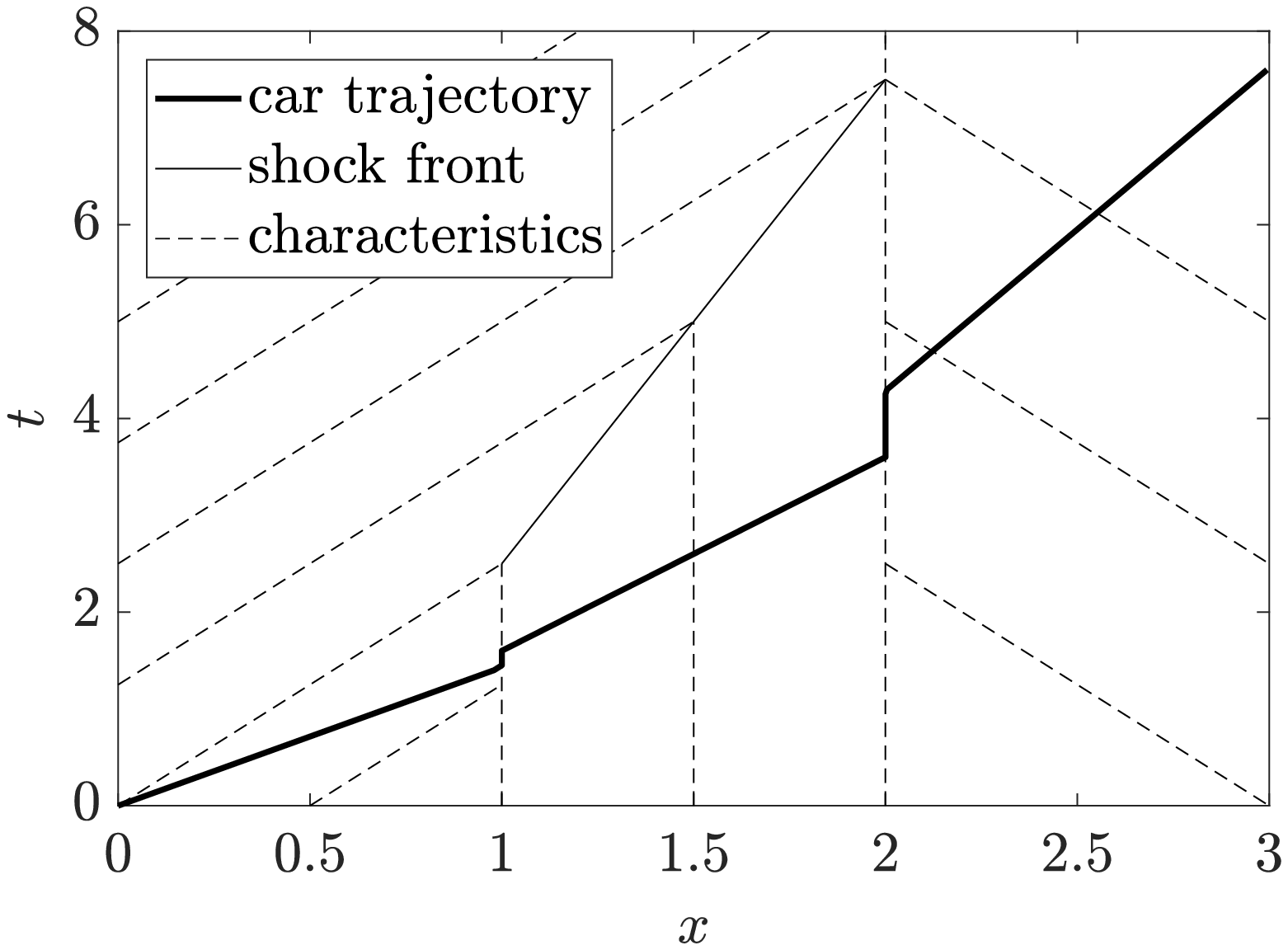}
	\end{minipage}
	\quad
	\begin{minipage}[b]{0.48\textwidth}
		\centering
		\includegraphics[width=\textwidth]{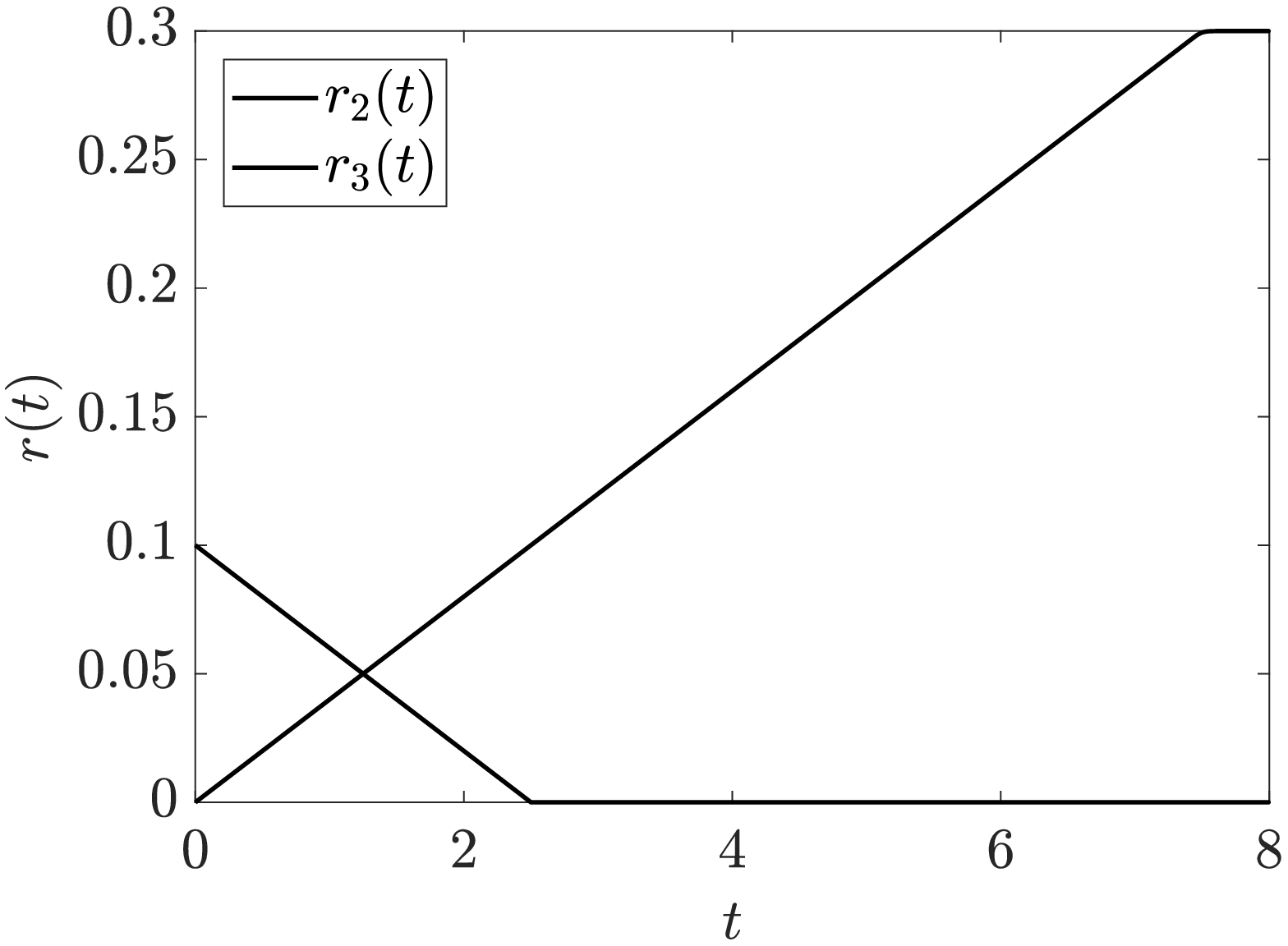}
	\end{minipage}
	\caption{Car trajectory and characteristics (left) as well as buffer loads (right)}
	\label{dia:Characteristics}
\end{figure}

For this small example, we are able to derive the analytical solution and hence the exact formulation of the car's trajectory can be described by
\begin{align*}
x(t) = \begin{cases}
0.7t & \text{if } ~0 \leq t \leq \tfrac{10}{7}, \\
1    & \text{if } ~ \tfrac{10}{7} \leq t \leq \tfrac{8}{5}, \\
1 + 0.5(t-1.6) & \text{if } ~ \tfrac{8}{5} \leq t \leq \tfrac{18}{5}, \\
2 & \text{if } ~ \tfrac{18}{5} \leq t \leq \tfrac{30}{7}, \\
2+0.3(t-\tfrac{30}{7}) & \text{if } ~ \tfrac{30}{7} \leq t \leq \tfrac{160}{21}. \\
\end{cases}
\end{align*}
The waiting times are $wt_2(\tfrac{10}{7}) \approx 0.17 $ at the second node and $wt_3(3.6) \approx 0.69 $ at the third one. Using the step sizes $h=0.1$ and $\tau=0.05$, we obtain for both algorithms a truncation error of 2.35e-14 which is in the range of the computational accuracy of MATLAB and the error is dominated by the round-off errors.
\\
\\
In a second example, we aim to study a non-linear case. More precisely, we compare two similar networks. In the first one, we consider a single road with length $L_1=2$ whereas the second network consists of two roads, $L_1=L_2=1$, with an empty buffer in between. Both settings are depicted in Figure \ref{im:Setting12}. 

\begin{figure}[H] 
	\centering
	\begin{tikzpicture}[decoration={markings,mark=at position 1 with {\arrow[scale=1,black]{latex}};}, scale=0.98]
	\coordinate[label=right:Setting 1:] ()  at (-2.5,2.6);
	\draw [-stealth,decoration={snake, amplitude = .4mm, segment length = 2mm, post length=0.9mm},decorate] (-2.5,1.5) -- (-0.4,1.5);
	\coordinate[label=right:$f_1^{in}(t) {=} 0.24$] ()  at (-2.5,1.85);
	\draw (-0.4,1.3) rectangle (0,1.7) node [pos=.5] {1};
	\draw [postaction={decorate}] (0,1.5) -- (4.5,1.5);
	\draw (4.5,1.3) rectangle (4.9,1.7) node [pos=.5] {2};
	
	\coordinate[label=right:Setting 2:] ()  at (5.4,2.6);
	\draw [-stealth,decoration={snake, amplitude = .4mm, segment length = 2mm, post length=0.9mm},decorate] (5.4,1.5) -- (7.5,1.5);
	\coordinate[label=right:$f_1^{in}(t) {=} 0.24$] ()  at (5.4,1.85);
	\draw (7.5,1.3) rectangle (7.9,1.7) node [pos=.5] {1};
	\draw [postaction={decorate}] (7.9,1.5) -- (9.85,1.5);
	\draw  (10.15,1.5) circle (3mm) node {2};
	\draw [postaction={decorate}] (10.45,1.5) -- (12.4,1.5);
	\draw (12.4,1.3) rectangle (12.8,1.7) node [pos=.5] {3};
	
	\draw [dashed] (1.45,2) -- (1.45,1);
	\draw [dashed] (2.25,2) -- (2.25,1);
	\draw [dashed] (3.05,2) -- (3.05,1);
	
	\draw [dashed] (9.35,2) -- (9.35,1);
	\draw [dashed] (10.15,2) -- (10.15,1.8);
	\draw [dashed] (10.15,1.1) -- (10.15,1);
	\draw [dashed] (10.95,2) -- (10.95,1);
	
	\coordinate[label=above:$\mathcal{C}_1$] ()  at (1.85,1.7);
	\coordinate[label=above:$\mathcal{C}_2$] ()  at (2.65,1.7);
	\coordinate[label=above:$x {=} 1$] ()  at (2.25,0.6);
	
	\coordinate[label=above:$\mathcal{C}_1$] ()  at (9.75,1.7);
	\coordinate[label=above:$\mathcal{C}_2$] ()  at (10.55,1.7);	
	\coordinate[label=above:$x {=} 1$] ()  at (10.15,0.6);
	
	\end{tikzpicture}
	\caption{Network with single road (Setting 1) and network with a buffer (Setting 2)}
	\label{im:Setting12}
\end{figure}
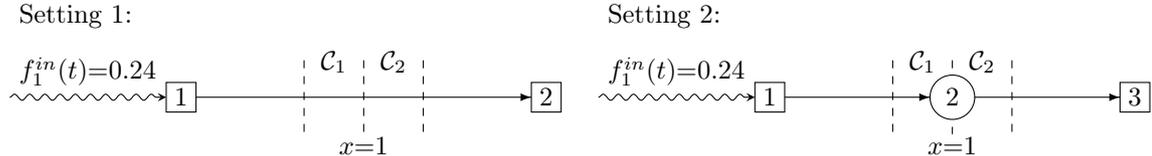

Now, we choose initial data such that a rarefaction wave travels through the junction and the car can pass the buffer immediately. In fact, we set
\begin{align*}
	\rho_{1,0}(x) = \begin{cases}
						0.4 & \text{if } ~x < 0.5\\
						0.2 & \text{if } ~x > 0.5
					\end{cases},~~ 
	\rho_{2,0} = 0.2 ~~\text{and}~~ r_{1,0} = r_{2,0} = r_{3,0} = 0.
\end{align*}
The remaining parameters are fixed to $r^{max}_1 = r^{max}_2 = r^{max}_3 = \infty$ and $\mu_{1} = \mu_{2} = 0.25$. Assuming that the car starts at the beginning of the first edge $x_1^*=0$ at time $t_1^*=0$, the exact solution for the car's path up to the end of the single road is given by
\begin{align*}
	x(t) = \begin{cases}
				0.6t & \text{if } ~0 \leq t<1.25,\\
				t-\tfrac{2 \sqrt{5}}{5} \sqrt{t}+0.5 & \text{if } ~ 1.25 \leq t \leq \tfrac{19+2\sqrt{34}}{10}.
		   \end{cases}
\end{align*}
In the second setting, the buffer remains empty and the total flow traverses the junction without being kept back. Consequently, we assume the same car trajectory as in the first setting. For step sizes $h=0.1 \cdot 2^{-n},~n=0,2,4,6$, we calculate the truncation error 
\begin{align*}
	\epsilon = \max_n \{ |x(t^n) - x^n| \},~t^n \in [0,T].
\end{align*}
The results are summarized in Table \ref{tab:Setting12}.

\begin{table}[h]
	\centering
	\setlength{\tabcolsep}{17pt}
	\begin{tabular}{ |c| |c|c|c|c|}
		\hline
		& \multicolumn{2}{c|}{Trunc. Err. (Setting 1)} & \multicolumn{2}{c|}{Trunc. Err. (Setting 2)}  \\
		\hline
		$n$ & Naive Algo & Complex Algo & Naive Algo & Complex Algo \\
		\hline
		\hline
		0 & 3.59e-02 & 4.14e-02 & 3.67e-02 & 4.17e-02 \\
		2 & 1.74e-02 & 1.83e-02 & 1.74e-02 & 1.84e-02 \\
		4 & 7.04e-03 & 7.29e-03 & 7.05e-03 & 7.30e-03 \\
		6 & 2.51e-03 & 2.58e-03 & 2.51e-03 & 2.58e-03 \\
		\hline
	\end{tabular}
	\caption{Truncation error $\epsilon$ for the settings depicted in Figure \ref{im:Setting12}}
	\label{tab:Setting12}
\end{table}

We observe that the resulting outcomes of the naive and complex algorithm are comparable. In both settings, the truncation error declines proportionally to the step sizes. The small difference between Setting 1 and 2 is due to behavior at $x=1$. In the first setting, the car can pass this point without any further calculations and at a speed corresponding to the density in the cell $\mathcal{C}_1$. In the second setting, the car drives exactly up to the end of the first road and checks whether the buffer is empty or not. Then, it starts on the second road at a speed corresponding to cell $\mathcal{C}_2$, which is higher than the speed corresponding to $\mathcal{C}_1$ since the traffic flow is described by a rarefaction. However, Table \ref{tab:Setting12} shows that the difference between both settings tends to zero for decreasing step sizes. In both settings, the error is reasonably small and evidence of convergence of the algorithms is observed.

\subsection{Comparison of Different Paths} \label{chap:Paths}
In this section, we consider a small network to study the different decision criteria of the car to choose the next edge. The network is depicted in Figure \ref{im:NetDec}. All roads are of length $L_e=1$.

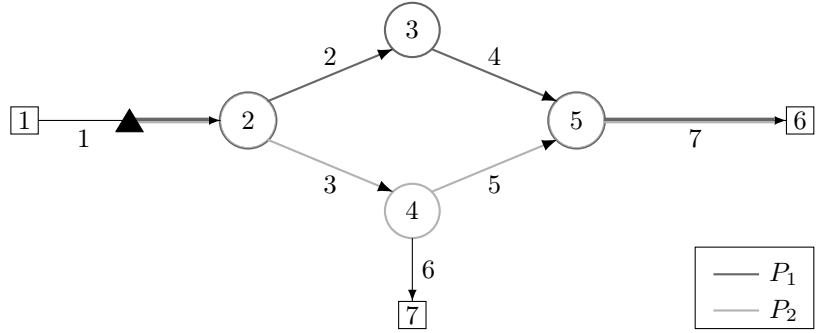
\begin{figure}[h] 
	\centering
	\begin{tikzpicture}[decoration={markings,mark=at position 1 with {\arrow[scale=1.1,black]{latex}};},scale=1.2]
	
	\draw (-0.3,-0.15) rectangle (0,0.15) node [pos=.5] {1};
	\draw [gray!120,thick] (2.3,0) circle (3.1mm);
	\draw [gray!60] (2.3,0) circle (3mm) node [black] {2};
	
	\draw [gray!120,thick] (4.1,1) circle (3mm) node [black] {3};
	\draw [gray!60,thick] (4.1,-1) circle (3mm) node [black] {4};
	\draw [gray!120,thick] (5.9,0) circle (3.1mm);
	\draw [gray!60] (5.9,0) circle (3mm) node [black] {5};
	\draw (8.2,-0.15) rectangle (8.5,0.15) node [pos=.5] {6};
	\draw (3.95,-2) rectangle (4.25,-2.3) node [pos=.5] {7};
	
	\draw [postaction={decorate}] (0,0) -- (2,0) node[draw=none,fill=none,near start,below] {1};	
	\draw [thick,gray!120,postaction={decorate}] (2.3+0.212,0.212) -- (4.1-0.212,1-0.212) node[black,draw=none,fill=none,midway,above] {2};
	\draw [gray!60,thick,postaction={decorate}] (2.3+0.212,-0.212) -- (4.1-0.212,-1+0.212) node[black,draw=none,fill=none,midway,below] {3};	
	\draw [thick,gray!120,postaction={decorate}] (4.1+0.212,1-0.212) -- (5.9-0.212,0.212) node[black,draw=none,fill=none,midway,above] {4};
	\draw [gray!60,thick,postaction={decorate}] (4.1+0.212,-1+0.212) -- (5.9-0.212,-0.212) node[black,draw=none,fill=none,midway,below] {5};
	\draw [postaction={decorate}] (4.1,-1.3) -- (4.1,-2) node[draw=none,fill=none,midway,right] {6};	
	\draw [postaction={decorate}] (6.2,0) -- (8.2,0) node[draw=none,fill=none,midway,below] {7};
	
	\draw [gray!120,line width=0.5mm] (1,0.01) -- (1.87,0.01);
	\draw [gray!120,line width=0.5mm] (6.2,0.01) -- (8.07,0.01);
	
	\draw [gray!60,thick] (1,-0.02) -- (1.87,-0.02);
	\draw [gray!60,thick] (6.2,-0.02) -- (8.07,-0.02);
	
	\draw[fill=black] (0.85,-0.125) -- (1.15,-0.125) -- (1,0.125) -- (0.85,-0.125);
	
	\draw (7.2,-1.4) rectangle (8.5,-2.3);
	\begin{scope}[shift={(7.4,-2.1)}]  
	\draw [thick,gray!60](0,0) -- plot[] (0.25,0) -- (0.5,0) 
	node[black,right]{$P_2$};
	\draw [thick,gray!120,yshift=0.4cm](0,0) -- plot[] (0.25,0) -- (0.5,0) 
	node[black,right]{$P_1$};
	\end{scope}
	
	\end{tikzpicture}
	\caption{Network together with the car's starting position (triangle) and the two possible paths $P_1$ and $P_2$ to reach the destination node $6$}
	\label{im:NetDec}
\end{figure}

The initial data are chosen such that the all fluxes are constant as long as the buffer load at junction 4 decreases. Therefore, we set a constant inflow $f^{in}_1(t)=0.2$ and the following initial densities
\begin{align*}
	\rho_{1,0} &= \tfrac{5-\sqrt{5}}{10},~ \rho_{2,0} = \tfrac{5-\sqrt{13}}{10},~ \rho_{3,0} = \tfrac{5-\sqrt{17}}{10},~ \rho_{4,0} = \tfrac{5-\sqrt{13}}{10}, \\
	\rho_{5,0} &= \tfrac{5-\sqrt{15}}{10},~\rho_{6,0} = \tfrac{5-\sqrt{10}}{10},~ \rho_{7,0} = \tfrac{5-\sqrt{3}}{10}.
\end{align*} 
All buffers, except at junction 4, are initially empty:
\begin{align*}
	r_{1,0} = r_{2,0} = r_{3,0} = r_{5,0} = r_{6,0} = 0 ~~\text{and}~~ r_{4,0}=0.5.
\end{align*}
The maximum buffer capacities are $r_v^{max}=0.5$, $v=2,3,4,5$, and infinity at the incoming and outgoing nodes. Additionally, we choose $\mu_v=0.25$ and distribution rates in the way that more cars prefer road 2 to road 3 and road 6 to road 5. More precisely, 
\begin{align*}
	\alpha_{2,1} = 0.6,~ \alpha_{3,1} = 0.4,~ \alpha_{5,3} = 0.4~ \text{and} ~\alpha_{6,3} = 0.6.
\end{align*}
The right-of-way priorities are not fixed but time-dependent on the demand and we consider a time horizon $T=15$.

We look at a car starting at different times $t_1^* \in \{0,5\}$ at the position $x_1^*=0.5$ (triangle) with destination node $6$. In Figure \ref{im:NetDec}, the two possible paths $P_1$ and $P_2$ are depicted. Figure \ref{im:CarTrajSmall} shows the distances traveled depending on the time calculated by the complex tracking algorithm with step size $h=0.01$. If the car starts at $t_1^*=0$, it must wait at junction $4$ since the buffer is not empty, $wt_{tot}^{P_2} \approx 0.78$. For this reason, the arrival times vary although the slope of the trajectories do not differ strongly. When the buffer is completely empty, the distribution rates ensure that the traffic is denser on the roads $e \in \{2,4\}$ than on the roads $e \in \{3,5\}$. This effect makes $P_2$ preferable when the car starts at $t_1^*=5$.

\begin{figure}[h]
	\centering
	\includegraphics[width=\textwidth]{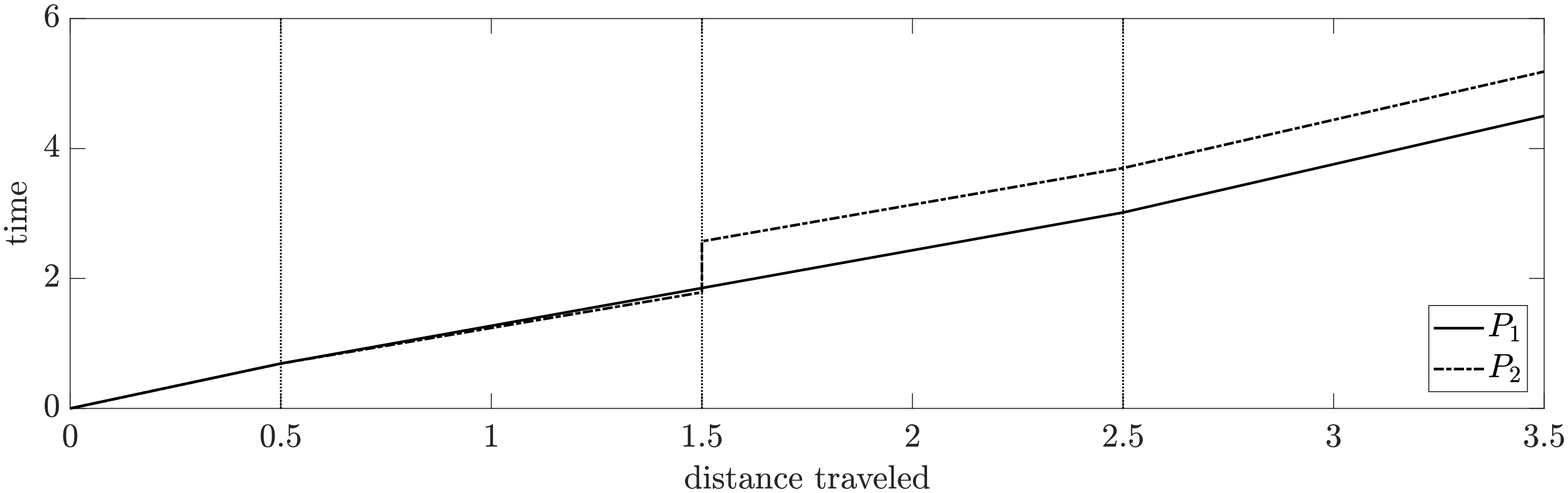}
	\includegraphics[width=\textwidth]{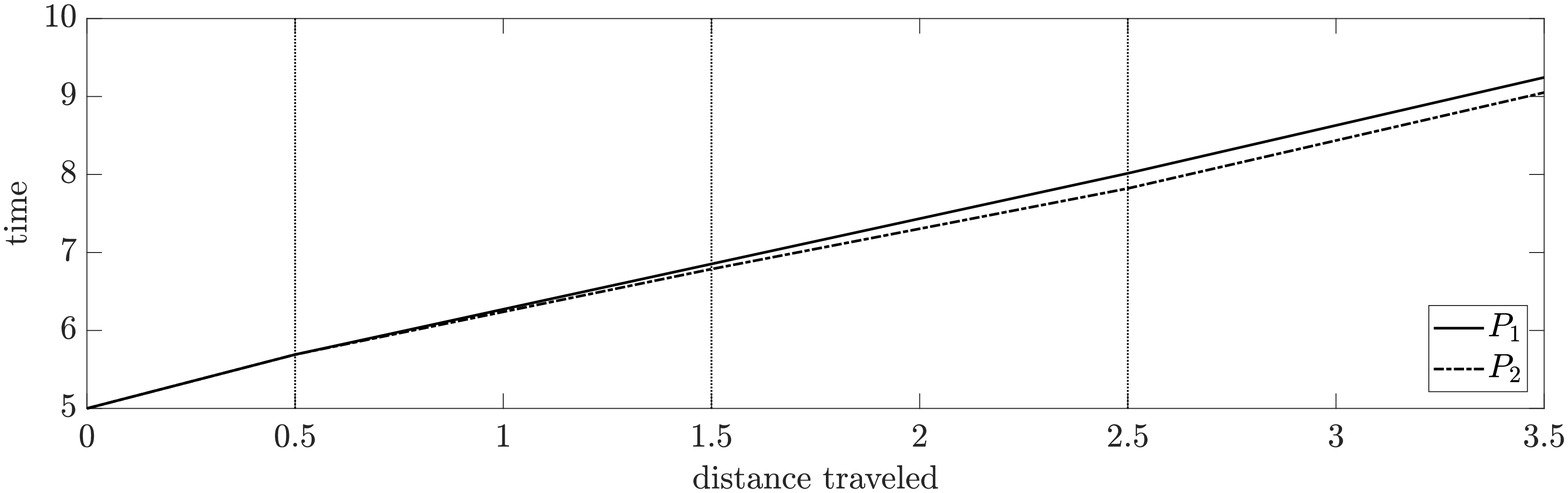}
	\caption{Distance traveled of a car following path $P_1$ and $P_2$ starting at the position $x_1^*=0.5$ at $t_1^*=0$ (top) and at $t_1^*=5$ (bottom)}
	\label{im:CarTrajSmall}
\end{figure}

Now, we study the choice of the path based on the different decision criteria introduced in Subection \ref{sec:CarChoice}. Since both paths are of length $L=4.5$, the cars are indifferent whenever their decision is completely based on the path length. The application of the time-dependent Dijkstra algorithm with complete information provides the optimum path related to the arrival time which is $P_1$ when the car starts at $t_1^*=0$ and $P_2$ when it starts at $t_1^*=5$. The same result is obtained assuming that accurate information on the densities is only available when the car is at a dispersing junction and when we successively update the path using the edge weights \eqref{eq:WeightD2} with $w^{\rho}=w^{r}=0.5$. Without knowing the exact departure time, the aggregated measure \eqref{eq:WeightsD1} with $w^{\rho}=w^{r}=0.5$ provides the path $P_2$. It includes the fact that from initial time $t_1^*>1$, following $P_2$ leads to an earlier arrival time than $P_1$.

\subsection{Tracking a Car in a Block Network} \label{chap:Block}

Now, we consider a bigger traffic network arranged in block structure and with road lengths of $L_e \in \{1,2\}$ as depicted in Figure \ref{im:BlockNet}. The network has a single inflow node with $f_1^{in}(t)=f(0.3)$ and two outflow nodes. We prescribe equal initial densities $\rho_{e,0}=0.3$ on all roads $e \in \cE$ and  $r_{v,0}=0$ at each junction $v \in \cV$. The maximum capacities at junctions inside the network are $r_v^{max}=0.3$ and the in- and outflow per time is restricted to $\mu_v=0.25$. The distribution rates are set to $\alpha_{i,j} = 0.5$ at all dispersing junctions and the priority parameters are chosen according to \eqref{eq:c1c2}. 

\begin{figure}[h]
	\centering
	\includegraphics[width=\textwidth]{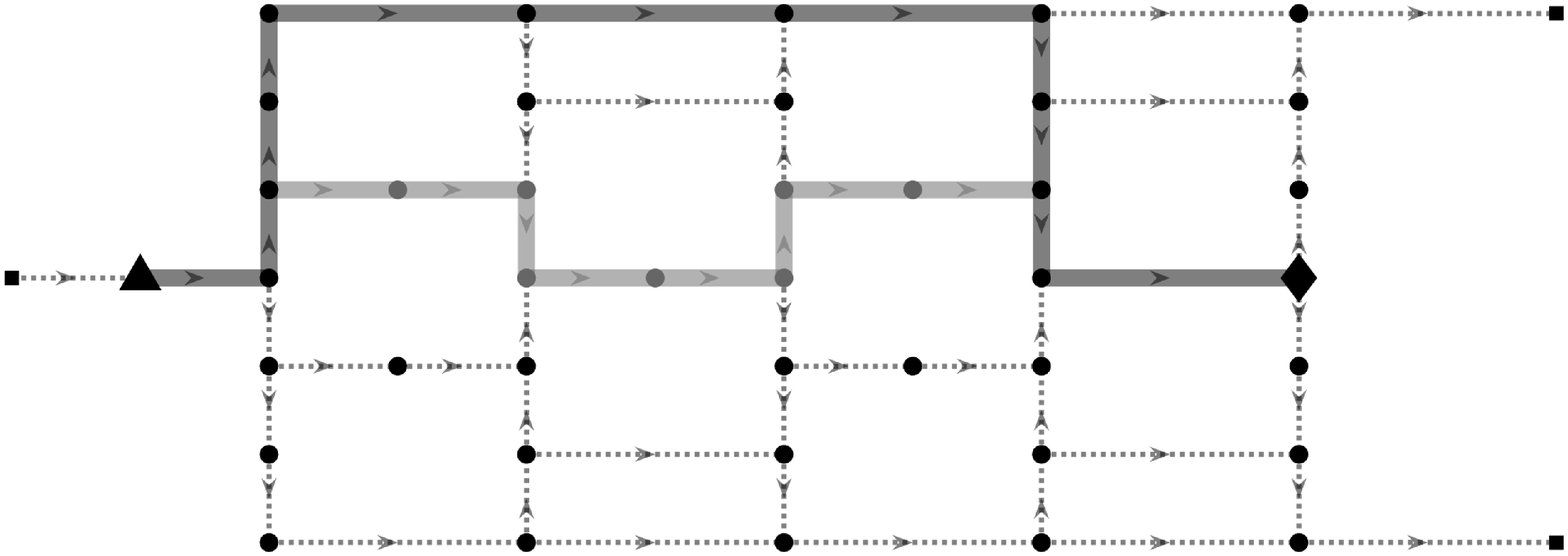}
	\caption{Block network with start (triangle) and destination node (diamond), shortest path (light gray) and fastest path (dark gray)}
	\label{im:BlockNet}
\end{figure}

The car starts at time $t_1^*=0$ at the position of triangle and wants to reach the destination marked by the diamond in Figure \ref{im:BlockNet}. Note that the network is symmetric with an axis of reflection that connects the triangle to the diamond. Hence, for our chosen initial data, we have two shortest and two fastest paths. One shortest path is visualized by the light gray line in Figure \ref{im:BlockNet}. It has a total length of $L^{short} = 13$ and the complex algorithm with $h_e=0.01$ yields an arrival time of $32.92$. However, a faster path can be found which is illustrated by the thick black line. Its total length $L^{fast}=15$ is higher but the car arrives earlier at time $t = 28.19$. 
 
\begin{figure}[h]
	\centering
	\includegraphics[width=\textwidth]{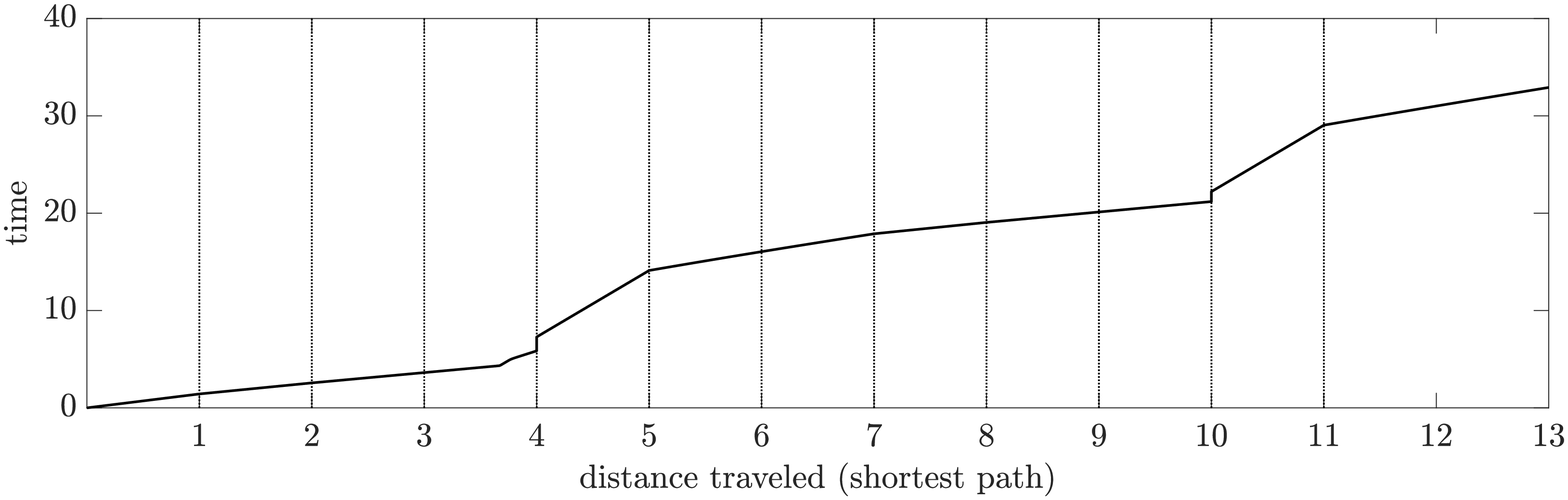}
	\includegraphics[width=\textwidth]{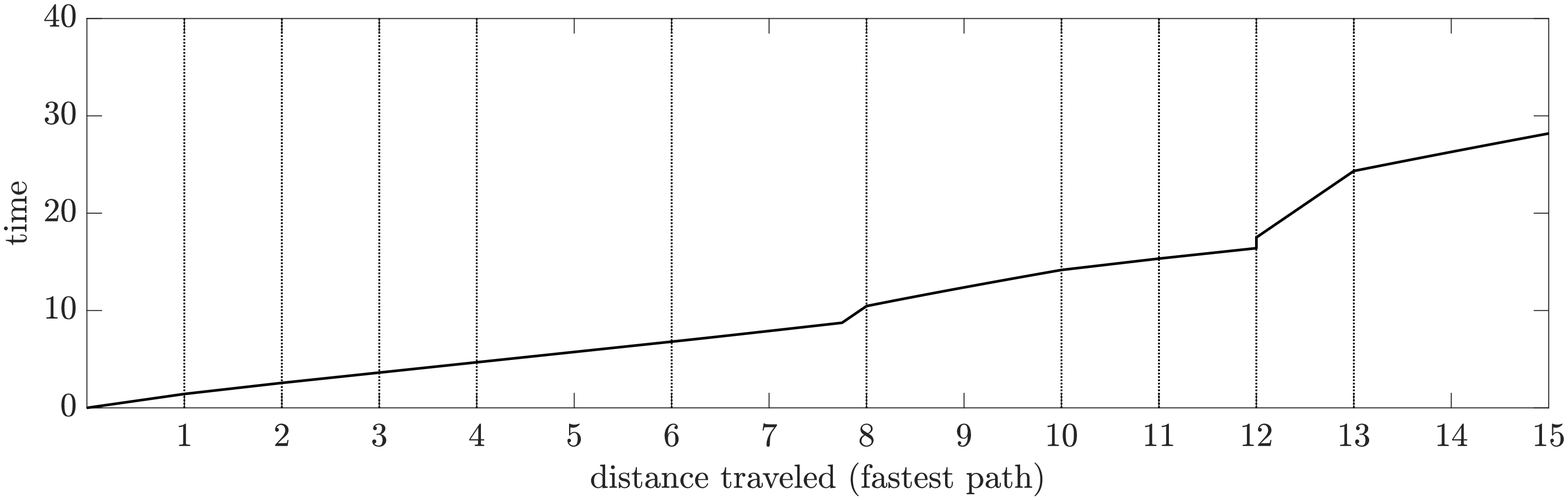}
	\caption{Distance traveled along the shortest and fastest path as depicted in Figure \ref{im:BlockNet}, passed nodes are indicated by vertical lines}
	\label{im:BlockTraj}
\end{figure}

In Figure \ref{im:BlockTraj}, the distance traveled depending on the time is depicted for two cars following the shortest and fastest path, respectively. Since the paths overlap at the beginning and at the end, it is not surprising that the slopes are similar in these areas. However, the slope in case of the shortest path is steeper in the middle section due to the higher traffic density in this part of the network. Additionally, the total waiting times differ. By following the shortest path, the car must wait at two junctions which leads to a higher waiting time of $wt^{short}_{tot} \approx 2.5$ instead of $wt^{fast}_{tot} \approx 1.1$ following the fastest path. 

The algorithm based on complete information finds the fastest path. However, for $w^{\rho} = w^{r} = 0.5$, the algorithms aiming to find an optimum path for situations with imperfect information, lead to the shortest path but do not provide the earliest arrival time.

\section{Conclusion} \label{chap:Conclusion} 

We considered the traffic flow model with buffers at junctions introduced in \cite{Herty.2009}. We particularly presented a new approach to define the buffer demand preventing negative buffer loads. The link between the original and the new formulation is described in Lemma \ref{lemma:dB}. 

Built on the concepts of \cite{Bretti.2008}, we successively extended two car path algorithms to networks with bounded buffers. Their numerical performance was evaluated by considering several test networks. Moreover, we investigated the possible paths and the corresponding arrival times at a given destination resulting from different choices to take the next road at junctions. 

Future work may include the investigation of car path tracking in second order traffic flow models on networks.

\section*{Acknowledgments}
This work was financially supported by BMBF project ENets (05M18VMA) and the DFG grant No. GO 1920/4-1.
\bibliographystyle{siam}
\bibliography{references}
\end{document}